\documentclass[11pt,a4paper,english]{article}
\usepackage{amsmath, latexsym, amsfonts, amssymb, amsthm, amscd}
\usepackage[hmargin=2.5cm,vmargin=2.5cm]{geometry}
\usepackage{pdfsync}
\usepackage{hyperref}
\usepackage{graphicx}
\usepackage{epsfig}
\usepackage{subfigure}
\usepackage{float}
\usepackage{authblk}
\usepackage{color}
\newtheorem{theorem}{Theorem}
\newtheorem*{theorem-non}{Theorem}
\newtheorem{corollary}{Corollary}
\newtheorem{proposition}{Proposition}
\newtheorem{lemma}{Lemma}

\newtheorem{assumption}{Assumptions}

\usepackage{setspace}
\usepackage{dsfont}

\newcommand{\eps} {{\epsilon}}

\newcommand{\Cov}{\mathrm{Cov}}

\DeclareMathOperator{\Var}{Var}

\newcommand{\bz}{{\bf z}}

\renewcommand{\phi}{\varphi}

\renewcommand{\P}{\mathbb P}

\newcommand{\E}{\mathbb E}
\newcommand{\R}{\mathbb R}

\newcommand{\N}{\mathbb N}

\newcommand{\ind}{1\!\kern-1pt \mathrm{I}}
\newcommand{\rsto}{]\!\kern-1.8pt ]}
\newcommand{\lsto}{[\!\kern-1.7pt [}

\newcommand\F{\mbox{I\kern-2pt F}}
\newcommand\1{\mathds{1}}

\newcommand\cC{{\mathcal C}}

\newcommand{\blue}{\textcolor[rgb]{0.00,0.00,1.00}}

\title{Asymptotics of the frequency spectrum for general Dirichlet  $\Xi$-coalescents }
\author[1]{Adri\'an Gonz\'alez Casanova}
\author[2]{Ver\'onica Mir\'o Pina}
\author[3]{Emmanuel Schertzer}
\author[4]{Arno Siri-J\'egousse}

\affil[1]{\footnotesize Instituto de Matem\'aticas de la Universidad Nacional Aut\'onoma de M\'exico, \'Area de la Investigaci\'on Cient\'ifica, Circuito Exterior, C.U., 04510 Coyoac\'an, CDMX, M\'exico.} 
\affil[2, 4]{\footnotesize Instituto de Investigaciones en Matem\'aticas Aplicadas y Sistemas, Universidad Nacional Aut\'onoma de M\'exico, Circuito Escolar 3000, C.U., 04510 Coyoac\'an, CDMX, M\'exico.}
\affil[2]{\footnotesize Centre for Genomic Regulation (CRG), The Barcelona Institute of Science and Technology, Barcelona, Spain.}
\affil[2]{\footnotesize Universitat Pompeu Fabra (UPF), Barcelona, Spain.}
\affil[3]{\footnotesize Faculty of Mathematics, University of Vienna,
Oskar-Morgenstern-Platz 1, 1090 Wien, Austria.}

\date{}

\begin{document}
\maketitle

\section*{Abstract}
In this work, we study general Dirichlet coalescents, which are a family of $\Xi$-coalecents constructed from i.i.d mass partitions, and are an extension of the symmetric coalescent. 
This class of models is motivated by population models with recurrent demographic bottlenecks. 
We study the short time behavior of the multidimensional block counting process whose $i$th component counts the number of blocks of size $i$. 
Compared to standard coalescent models (such as the class of $\Lambda$-coalescents coming down from infinity), our process has no deterministic speed of coming down from infinity. In particular,
we prove that, under appropriate re-scaling, it converges to a stochastic process which is the unique solution of a  martingale problem.
We show that the multivariate Lamperti transform of this limiting process is a Markov Additive Process (MAP).
 This allows us to provide some asymptotics  for the 
 $n$-Site Frequency Spectrum,  which is a statistic widely used in population genetics. 
 In particular, the rescaled number of mutations converges to the exponential functional of a subordinator.

\section{Introduction} 
\subsection{General Dirichlet $\Xi$-coalescents}

Coalescents with simultaneous and multiple collisions ($\Xi$-coalescents, \cite{S2000,mohle3,BertoinLeGall}) are exchangeable coagulating continuous-time Markov chains with values in the set of partitions of $\N$.
They can be built via a paintbox construction (see  \cite{Berbook}, Section 4.2.3).
Consider a finite measure  $\Xi$ on  $\Delta=\{\vec p=(p_1,p_2,\dots): \sum p_j\le1\}$, the simplex on $(0,1)$, and a Poisson point process (Ppp) on $\R_+\times \Delta$, of intensity $dt\otimes\Xi(d\vec p)/\sum p_j^2$, giving the coagulation times and rules.
At every jump time, (1) consider a family of disjoint subintervals $(I_1,I_2,\dots)$ of $(0,1)$, where  $I_j$ has length $p_j$, (2) throw every block present at this time uniformly at random and independently  in $(0,1)$, and (3) blocks falling into the same subinterval among $(I_1,I_2,\dots)$ merge into one.

In this work, we consider a particular subclass of $\Xi$-coalescents where the interval partition of the paintbox has a generalized version of a Dirichlet distribution with a random number of components.
More precisely, consider a sequence of non-negative numbers $(R(k); k\in\N)$ and $m$ a probability measure on $(0,\infty)$. Then, generate at rate $R(k)$ a partition $(p_1^{(k)},\dots,p_k^{(k)})$ where
\begin{equation}\label{masspart}
\forall j\in[k], \ \ p_j^{(k)} \ := \ \frac{w_j}{s_k}
\end{equation} 
where $(w_1,\dots,w_k)$ are i.i.d. random variables with law $m$ on $(0,\infty)$, \blue{$s_k := \sum_{i=1}^k w_i$ and $[k]:=\{1,\dots, k\}$}. As in the previous paintbox construction, blocks are assigned a uniform random variable, and we merge all the blocks falling in the same interval.
\blue{This corresponds to a $\Xi$-coalescent where the characteristic finite measure $\Xi$ on the infinite simplex $\Delta$ is described as follows. For every $k\in \N$, let us define $\nu_k$, a probability measure on the infinite simplex $\Delta$, s.t. $$ \nu_k = {\cal L}((w_1/s_k, w_2,s_k, \dots, w_k/s_k, 0,0,\dots )),$$
where ${\cal L}(X)$ denotes the law of the random variable $X$. Then, for every measurable $B \subset \Delta$, 
$$\Xi(B) = \sum_{k=1}^\infty \int_B R(k) \left( \sum_{i=1}^\infty x_i^2 \right) \nu_k(d(x_1, x_2, \dots)). $$
} 
The case where the $w_i$'s are Gamma distributed corresponds to the standard Dirichlet mass-partition. \blue{In particular, if the $w_i$'s are exponentially distributed, it corresponds to a symmetric Dirichlet distribution.}
We refer to this model as the general Dirichlet coalescent. 
The name Dirichlet coalescent was coined in \cite{GM}. Therein the authors consider paintbox construction according to a Dirichlet distribution with a fixed number of components. 
 
Another example of such a process is the symmetric coalescent defined in \cite{GCMPSJ}, which corresponds to the case where $w_i=1$ a.s..
In that case, in order to correspond to the paintbox construction described above, the sequence $R$ must satisfy that $\sum R(k)/k<\infty$, see \cite{GCMPSJ}.
Here, we assume a finite second moment for $m$ and that the rate of $k$-events has a heavy tail, in the following sense.
\begin{assumption}\label{def:finite-xi}
\begin{enumerate}
\item There exist $\alpha\in(0,1)$ and $\rho>0$ such that  $R(k) k^\alpha\to \rho$ as $k\to\infty$.
\item $\int_{\R_+} x^2 m(dx) <\infty$.
\end{enumerate}
\end{assumption}
In fact, a $\Xi$-coalescent is well defined if the rate at which two blocks merge into one is finite \cite{S2000}. It is easy to see from the paintbox construction that, since the vector $(p_1^{(k)},\dots,p_k^{(k)})$, is exchangeable, this rate is given by
$$\sum_{k \ge 1} R(k) k \E((p_1^{(k)})^2), $$
which is finite under Assumptions \ref{def:finite-xi} (see Proposition \ref{prop:moments} in Appendix A).

Suppose that the general Dirichlet coalescent starts with $n$ singletons.  
 Denote by $\hat \mu_t^n=(\hat \mu_t^n(1),\dots, \hat \mu_t^n(n))$ \blue{the vector such that}  $\hat \mu_t^n(i)$ is the number of blocks containing $i$ elements at time $t$, and denote by $|\hat \mu_t^n|$ the total number of blocks.
Define the rescaled vector
 \begin{equation}\label{muresc}
 \mu_t^n = \frac1n \hat \mu^n_{ t{n^{\alpha-1}}}.
 \end{equation}
In this paper, we aim at studying the limiting behavior of the Markov process $(\mu_t^n;t\geq0)$ as $n \to \infty$. We prove (in Theorem \ref{thm:conv-process}) that it converges towards a stochastic process $(\mu_t; t\geq 0)$, defined as the unique solution to a martingale problem associated to a continuous coagulation operator (see Theorem \ref{Mgprob}). 

\blue{Intuitively, the result can be understood as follows.
In the paintbox construction, when there are $n$ lineages, a $k$-merging event corresponds to throwing $n$ balls into $k$ boxes (with probabilities $(p_1^{(k)},\dots,p_k^{(k)})$), and merging the balls that land in the same box. 
For $k\gg n$, the chance that non-trivial merging occurs is negligible, whereas for $k\ll n$, all lineages will be merged into a few lineages (which disappear when rescaling the number of blocks by $n$). 
The total rate of $k$-events with $\epsilon n \le k \le M n$ for some small $\epsilon >0$ and large $M<\infty$ can by approximated by $\int_{\epsilon n}^{M n} \rho y^{-\alpha}dy = C n^{1-\alpha}$ with some constant $C$, which explains why time is slowed down by this factor. The heuristics behind the form of the coagulation operator that is the central part of the generator of the limit process are explained in Section \ref{heuristics}.
}

We also show that this limit process is self-similar with negative index $\beta:= \alpha -1$ (see Theorem \ref{thm:lamperti}).
In particular, the limit of the \blue{rescaled} block counting process $(|\mu_t|; t\geq 0)$ is the exponential of a time-changed subordinator. As a direct corollary, if we define
\begin{eqnarray}\label{eq:time-change} A_t \ := \ \inf\{s>0 : \ \int_{0}^s \frac{1}{|\mu_r|^{\beta}} dr >t \}, \ \ \mbox{and} \  \xi_t \ := \ -\log\bigg( |\mu_{A_t}| \bigg),\end{eqnarray}
then $(\xi_t; t\geq 0 )$ is a subordinator. \blue{The law of the subordinator can be identified as a direct consequence of Theorem \ref{thm:flow-representation} (see Section \ref{sect:selfsimilarity}).}
This shows that the short time behavior of the block counting process remains stochastic. This is in sharp contrast with previous studies where it is shown that classical models (such as $\Lambda$-coalescents) exhibit a deterministic  speed of coming down from infinity, see Section \ref{sect-speedofcdi} for a more detailed discussion. 
Our result can be interpreted as a stochastic speed of coming down from infinity.

\subsection{Speed of coming down from infinity}
\label{sect-speedofcdi}

We say that a $\Xi$-coalescent comes down from infinity if there are finitely many blocks  at any time $t>0$ almost surely, even if the coalescent is started with infinitely many blocks. 
In his original work, Schweinsberg already established a criterion for coming down from infinity \cite{S2000}.
In the case that the characteristic measure of the coalescent is supported on the set of finite mass partitions $\Delta^*=\{\vec p=(p_1,\dots, p_k): \sum_{j=1}^k p_j= 1, \text{ for some } k\}$  
(which is our case of interest), the process comes down from infinity if and only if
\[\int_{\Delta} \frac{\Xi(d\vec p)}{\sum_jp_j^2} = \infty.\]
Otherwise, the number of blocks stays infinite for a finite amount of time.
As another example, coalescents whose characteristic measures are supported only on infinite mass partitions, i.e. for which $\Xi(\Delta^*) = 0$, either come down from infinity or always stay infinite. 
\blue{We consider coalescents supported on $\Delta^*$ and that come down from infinity.} Limic \cite{limic} studied the small time behavior of $\Xi$-coalescents
under what she called a regularity assumption 
\[\int_{\Delta} \frac{(\sum_j p_j)^2}{\sum_j p_j^2}\Xi(d\vec p) < \infty. \]
In this setting, and starting with infinitely many lineages, there exists a speed of coming down from infinity, i.e., a deterministic function $\nu_\Xi(t)$, which is finite for all $t>0$ such that, if $|\hat\mu_t|$ is the number of blocks at time $t$ in a coalescent starting with infinitely many lineages, 
\[\lim_{t\to0^+} \frac{|\hat\mu_t|}{\nu_\Xi(t)} = 1, \textrm{ almost surely.}\]
This mirrors the behavior of the class of $\Lambda$-coalescents coming down from infinity \cite{BBS, BBL}. To summarize, most of the previous studies have shown that the block counting process of a large class of exchangeable coalescents exhibits a deterministic behavior at small time scale.

In the present work, we consider $\Xi$-coalescents belonging to the first family (for which $\Xi$ is supported on $\Delta^*$), and which come down from infinity \blue{(see \cite{S2000}, Section 5.5)}. We take a different approach, since we study the rescaled number of blocks, starting from $n$ lineages, as $n\to \infty$. In our case, when time is re-scaled by $n^{\alpha-1}$, the block counting process converges to a {\it stochastic}  self-similar process, so there is no deterministic speed of coming down from infinity.

Our results have similar flavor to those of Haas and Miermont \cite{HM} for $\Lambda$ coalescents with dust, and of M\"ohle and co-authors \cite{ GM, mohle2} for  a class of $\Xi$-coalescents with dust.
In the first work, a self similar behavior of the rescaled number of blocks is obtained in the limit. 
In the second work, they prove that the frequency of singletons, as well as the number of blocks rescaled by $n$, converges to the exponential of a subordinator (without any time-rescaling).
A natural prospect of research would be to identify conditions that would partition $\Xi$-coalescents (coming down from infinity)  into two main classes: a first class  with a deterministic limiting behavior, and a second one with a stochastic descent from  infinity.

\subsection{Perspectives on coming down from infinity}
Our results deal with processes valued on the partitions of $n$ when $n$ goes to infinity. Although this is heuristically related to the case $n=\infty$, which corresponds to working with partitions of $\N$, we expect that there are important technical challenges when studying the process starting with infinitely many blocks. To be precise, the latter would require an entrance law at infinity for the limit of the multidimensional block counting process. In our approach, we avoid this problem by rescaling the block counting process by $n$ so that $|\mu_0^n| = 1$ and there is no need for entrance laws for  the limit process $(\mu_t; t\geq 0 )$.  

The study of entrance laws of self-similar Markov processes has recently been an active area of research. In the one dimensional case there is an extensive literature (see for example \cite{S1,S2, S3} and the references therein). Recent results in the finite dimensional case can be found in \cite{S5}. This is also a classic problem for Markov additive processes \cite{S4}. 
We believe that our results can motivate the study of entrance laws for infinite dimensional self-similar processes.

\subsection{Biological motivation}
   The symmetric coalescent \cite{GCMPSJ} can be obtained as the limiting genealogy of a Wright-Fisher population that undergoes rare recurrent bottlenecks 
 reducing the population size to a random number $k$ of individuals for only one generation. 
In this case, the second point of Assumptions \ref{def:finite-xi} always holds and the first point is fulfilled if the measure characterizing the size of the bottlenecks has power tails of order $\alpha$.
General Dirichlet coalescents naturally arise in an extension of this model, that can be seen as multinomial non-exchangeable reproductive events \cite{exchnonexch}.

The analysis of the asymptotics of the multidimensional block counting process
 allows us to characterize the limiting behavior of the Site Frequency Spectrum (SFS) \blue{of} our family of $\Xi$-coalescents as some functional of the limit process $(\mu_t; t\geq 0)$.
The SFS is one of the most widespread statistics in population genetics. It consists in a vector of size $n-1$ whose $i$th component counts the number of mutations that are shared by $i$ individuals in a sample of size $n$. We suppose that mutations occur at a constant rate over the coalescent tree started with $n$ individuals, so that the SFS is closely related  to its branch lengths. In general, this is a complex combinatorial problem and most of the previous works have relied on some approximations of the short time behavior of the block counting process to derive asymptotics for the lower part of the SFS (i.e., number of singletons, pairs etc.). 
Some examples are  \cite{ DKk} for the case of the Kingman coalescent, and  \cite{BBS, BDDK} for coalescents with multiple collisions, such as Beta-coalescents, or \cite{BG, KPSJ, DKbs, KSJW} for the special case of the Bolthausen-Sznitman coalescent.
For fixed $n$, some studies on the law of the SFS can be found in \cite{FL, HSJB, KSJW}.

There are few results available regarding the SFS of $\Xi$-coalescents. Works like \cite{BCEH,spence} present computational algorithms based on recursions to derive the expected SFS for finite $n$. 
Asymptotic properties of $\Xi$-coalescents started with $n$ lineages were studied previously, in particular regarding the number of blocks \cite{Moh10}. Theorem \ref{thm:cv-sfs} describes the asymptotics of the SFS for general Dirichlet coalescents.

\section{Main results}\label{mainres}

\subsection{Notation}

Let us start this section with some notations.
We denote by $\N$ the positive integers and by $\mathbb{N}_0$ the non-negative integers.
Let $\ell^1(\R_+)$ be the set of all sequences with positive \blue{coordinates} and with finite sum. For every  $\bz=(z(1), z(2),\dots)\in\ell^1(\R_+)$,  we denote the sum of all its elements by $|\bz|= \sum_{i =1}^{\infty} z(i)$.
We also denote by $\ell^1(\N_0)$ the set of sequences with coefficients valued in $\N_0$ and finite sum.
Define \[{\cal Z} := \left\{ {\bf z}\in \ell^{1}(\R^+) \ : \ \sum_{i=1}^\infty i z(i) =1 \right\}, \ \ {\cal Z}_n \ := \ \{\bz \in{\cal Z}  \ : \ n \bz\in \ell^1(\N_0)   \}.\]
The space ${\cal Z}$ will be equipped with the $\ell^1(\R_+)$ norm.
The latter definitions are motivated by partitions of $n\in\N$. Recall that a partition of $n\in\N$ denotes an unordered sequence of integers  $\{m(1),\dots,m(k)\} $ such that $\sum_{i=1}^k  m(i) =n$.
For every $i\in\N$, define $z({i}) = \frac{1}{n} \#\{k: m(k) =i\}$, the ``frequency'' of $i$ in the partition of $n$.  Then ${\bf z}=(z(1), z(2),\dots)$ is an element of ${\cal Z}_n$. Starting with $n$ singletons, we study the multidimensional block counting process of the general Dirichlet coalescent as a Markov process valued in ${\cal Z}_{n}$ (as already outlined in the introduction). Its first component denotes the frequency of singletons, its second component is the frequency of pairs, etc.

For every $c=(c(1),c(2),\dots)\in \ell^1(\N_0)$,
we define $\phi(c)=\sum_{i=1}^\infty i c(i)$ so that for every $\ell\in \N$,   
\[\phi^{-1}(\ell) := \{ c\in \ell^1(\N_0) \ : \ \sum_{i=1}^\infty i c(i) =\ell \}.\] 
Observe that $\phi^{-1}(\ell)$ is a way of encoding partitions of $\ell$. 
For any partition $c\in \phi^{-1}(\ell)$, we also define $c ! \ = \ \prod_{i=1}^\ell c(i) !$.

For any random variable $X$ valued in $\ell^1(\R_+)$, we denote by $\E(X)$ the deterministic vector obtained by taking the expected value coordinatewise. 
For $T>0$, we denote by $D([0, T], {\cal Z} )$ the space of c\`adl\`ag functions on $({\cal Z}, \ell^{1}(\R_+))$ equipped with the Skorokhod $M_1$ topology \cite{Skorokhod,Whitt}.

For any $ \lambda\in[0,1]$, we define 
\[ \forall {\bf z}\in{\cal Z}, \   \psi_{\lambda}({\bf z}) \ := \ \sum_{i=1}^\infty z(i) \lambda^i,\]
and for every $\vec{\lambda}\in [0,1]^{K},  K \in \N$, we define 
\[ \forall {\bf z}\in{\cal Z}, \ \ \psi_{\vec{\lambda}} ({\bf z}) : = (\psi_{\lambda_1}({\bf z}), \dots, \psi_{\lambda_K}({\bf z})).\]
 We define the following set of test functions
\begin{eqnarray*}
\mathcal{T} := \bigg \{ f: {\cal Z} \to \R:   \exists \vec{\lambda} = (\lambda_1, \dots, \lambda_K) \in [0,1]^{K} , F \blue{\textrm{ is Lipshitz continuous on }[0,1]^K}, \\ \textrm{s.t. } f(\bz) \ = \ F\left(\psi_{\lambda_1}(\bz),\dots,  \psi_{\lambda_K}(\bz)\right) =
F\circ\psi_{\vec{\lambda}}(\bz)  \bigg \}.
\end{eqnarray*}

\subsection{Convergence of the rescaled partition process}\label{heuristics}

We are now ready to enunciate and comment on our main results.
We start by describing the random coagulation corresponding to the jump events of the general Dirichlet $\Xi$-coalescent.
\blue{Set  $\bz_n\in {\cal Z}_n$ where $n \bz_n(i)$ is the number of balls of size i (the size of a ball refers to the number of samples/lineages it represents) .} Then throw
 $n \bz_n(i)$ balls of size $i$, $i\geq1$, at random into $k$ boxes in such a way that the probability of falling into box $j\in[k]$ is $p^{(k)}_j$, as defined in \eqref{masspart}. 
Now, define  $\Lambda^{k,n}({\bf z}_n)\in {\cal Z}_n$ as
\begin{equation}\label{defLkn}\forall \ell\in\N,  \ \Lambda^{k,n}({\bf z}_n)(\ell) \ = \ \frac{1}{n} \#\{ j \leq k: \ \mbox{sum of the sizes of the balls falling in box $j$ is $\ell$} \}. \end{equation}
Note that, for $\ell \ge n,  \Lambda^{k,n}({\bf z}_n)(\ell) = 0$. 
By a slight abuse of notation, we define the (random) operator $\Lambda^{k,n}$ acting on ${\cal Z}_n$ such that, for
every function $g$ defined on ${\cal Z}_n$,
\[  \Lambda^{k,n}g ({\bf z}_n) \ = \  g( \Lambda^{k,n}({\bf z}_n) ). \]
Thanks to these notations we can define the infinitesimal generator of the ${\cal Z}_n$-valued process 
$(\mu_t^n; t\geq0)$ defined in \eqref{muresc} as
\begin{equation}
\label{generator-discrete}   \mathcal{A}_n f( \bz_n ) \ = \ \frac{1}{n^{1-\alpha}} \ \sum_{k\geq 1} R(k)\bigg( \E\left( \Lambda^{k,n}f(\bz_n ) \right) - f(\bz_n) \bigg)
\end{equation}
for every measurable and bounded $f: {\cal Z}_n \to \R$ and  $\bz_n\in{\cal Z}_n$.

Before diving into technicalities, let us first motivate the coming results. Assume that $k,n\to\infty$ with $ k/n\sim x\in(0,\infty)$, i.e., a large number of balls ($n$) and boxes ($k$), of the same order. Under this restriction, if an event involving $k$ boxes occurs, the number of balls of size $i$ falling in box 1 is well approximated by a Poisson random variable with parameter 
$$
p_1^{(k)}  n z_n(i) \approx  \frac{\Gamma}{x} z_{n}(i),
$$
where $$
\Gamma :=\frac{ w_1}{\E(w_1)}.
$$
Further, since the number of balls/boxes is large,  the total number of boxes with $r$ balls of \blue{size} $i$ should be well approximated by 
$$
k \times \left( \frac{\Gamma}{x} z_{n}(i)\right)^{r} \frac{e^{- \frac{\Gamma}{x} z_{n}(i)}}{r !}. 
$$ 
By a similar heuristic, \blue{if $k,n\to\infty$ with $ k/n\sim x\in(0,\infty)$}, we expect
\begin{equation*}
 \Lambda^{k,n}({\bf z}_n)(\ell)  \approx  \E\left( x \exp(-{|\bz_n|\Gamma}/{x}) \sum_{c \in \phi^{-1}(\ell)}   \prod_{i=1}^\ell  \frac{(z_n(i)\Gamma/x)^{c(i)}}{c(i) !}\right),
\end{equation*}
where the expectation is taken with respect to the random variable $\Gamma$. This justifies the limit operator introduced later on in \eqref{defCx}.

\subsection{A martingale problem}
\label{sect-Mgprob}
We now define  the martingale problem, associated to a continuous coagulation operator, described as follows. 
Let $x>0$ and define  ${\cal C}^{x}: {\cal Z} \to {\cal Z}$, such that its $\ell$th coordinate is given by
\begin{equation}\label{defCx}
{\cal C}^{x}(\bz)(\ell)=\E\left( x \exp(-{|\bz|\Gamma}/{x}) \sum_{c \in \phi^{-1}(\ell)}   \prod_{i=1}^\ell  \frac{(z(i)\Gamma/x)^{c(i)}}{c(i) !}\right). \end{equation}
From the heuristics of the previous section, ${\cal C}^{x}(\bz_n)(\ell)$ is a natural candidate to approximate $
 \Lambda^{k,n}({\bf z}_n)(\ell) $.
As for $\Lambda^{k,n}$, we define the operator $\cC^x$ \blue{on functions on} on ${\cal Z}$ such that, for
every function $g$ bounded and measurable on ${\cal Z}$,
\[  \cC^xg(\bz) \ = \  g( \cC^x(\bz) ). \]
We will show in due time that $\cC^x(\bz)\in{\cal Z}$, see Proposition \ref{prop1}.

\begin{theorem}[Uniqueness of the martingale problem]\label{Mgprob}
For every $\bz \in {\cal Z}$ and $f \in {\mathcal T}$, the function 
\begin{equation*}
x\to  \bigg( \cC^xf(\bz) -f(\bz) \bigg) \rho x^{-\alpha}
\end{equation*}
is integrable on $(0,\infty)$. Define 
\begin{equation}\label{genA} {\cal A}f(\bz) \ := \ \int_0^\infty \bigg( \cC^xf(\bz) -f(\bz) \bigg) \rho x^{-\alpha}dx. \end{equation}
 There exists a unique c\`adl\`ag process $(\mu_t; t\geq0)$ valued in ${\cal Z}$  with $\mu_0=\bz$ such that  
\begin{equation}
 \left( f(\mu_t) - \int_0^t {\mathcal A} f(\mu_s) ds;  \ t\geq0\right) \ \ \ \mbox{is a martingale for every $f\in{\cal T}.$}
\label{MP}
\end{equation} 
\end{theorem}

 Theorem \ref{Mgprob} characterizes the limiting process in the following result.
 
\begin{theorem} \label{thm:conv-process}
Suppose that Assumptions \ref{def:finite-xi} hold.
If $\mu_0^n=\bz_n \to \bz \in{\cal Z}$,
then for every $T>0$, 
\begin{equation*}
(\mu_t^n; t\geq0)  \Longrightarrow(\mu_t; t\geq0) \textrm{ in $D([0, T], {\cal Z})$}, 
\end{equation*}
where the process $(\mu_t; t\geq0)$ is the unique solution to the martingale problem \eqref{MP} with initial condition ${\bf z}$.
\end{theorem}

\subsection{Self-similarity}

The second part of this paper is devoted to the study of the limiting process $(\mu_t; t\geq0)$ characterized in Theorem \ref{Mgprob}. We prove that it is an infinite dimensional self-similar process, with negative index $\beta:= \alpha - 1 \in(-1,0)$ (Proposition \ref{prop:self-similar}). We can characterize its infinite dimensional Lamperti-Kiu transform.

\blue{The fact that $(\mu_t; t\geq0)$ is self-similar is inherited from the regular tail behavior of $R(k)$ (which is reflected by the $x^{-\alpha}$ in the generator (equation \eqref{genA}) together with the fact that for any positive constant $\gamma$ and for every $\ell \in \N$, 
$\gamma \cC^{x}( \bz)(\ell)  =  \cC^{\gamma x} ( \gamma\bz)$  (see equation \eqref{defCx}). 
}

To characterize this transformation, first consider 
the limiting block counting process $(|\mu_t|;t\geq0)$. According to Proposition \ref{prop:self-similar},  
$(|\mu_t|;t\geq0)$ is a non-increasing self-similar positive Markov process with parameter $\beta$.
The standard Lamperti transform tells us that such a process is identical in law to the exponential of a time-changed subordinator. 
Recall $(A_t ;t>0)$ and $ (\xi_t  ;t>0)$ defined in \eqref{eq:time-change}.
Then $(\xi_t; t\geq 0 )$ is a subordinator
 and $(|\mu_t|; t\geq0)$ can be recovered 
by the relation $ |\mu_t| =  \exp(-\xi_{\tau_t})$, where $\tau_t:= \inf\{u>0 \ : \ \int_0^u \exp(\beta \xi_s) ds>t\}$.

Let us turn to the infinite dimensional self-similar process $(\mu_t; t\geq0)$. Let $S := \{ \bz\in \ell^1(\R_+) : |\bz|=1 \}$
be the unit sphere in $\ell^1(\R_+)$. 
The idea  of the infinite dimensional Lamperti-Kiu transform is to decompose the process into its ``radial part'' $(|\mu_t|; t\geq0)$ (the block counting process) and its ``spherical part'' 
\[\bigg(\left(\frac{\mu_t(1)}{|\mu_t|}, \frac{{\mu_t(2)}}{{|\mu_t|}}, \dots \right); t\geq0 \bigg) \in S,\] which encodes the evolution of the asymptotic frequencies of singletons, pairs, etc.
In the spirit of the one-dimensional case, the process can be related to a time-changed Markov additive process (MAP, see \cite{cinlar}).
Theorem \ref{thm:lamperti}  is a natural extension of Theorem 2.3. in  \cite{Alili}, established in finite dimension.

\begin{theorem}[Lamperti-Kiu transform]\label{thm:lamperti}
For every $t>0$, define the stopping time $(A_t; t\geq0)$ as in (\ref{eq:time-change})
and $\bigg((\xi_t,\theta_t); t\geq0\bigg)$ as the process valued in $\R_+\times S$ such that
\[ \xi_t \ := \ -\log\bigg( |\mu_{A_t}| \bigg), \ \ \theta_t \ := \ \frac{1}{|\mu_{A_t}|} \mu_{A_t}. \]
Then, $\bigg((\xi_t,\theta_t); t\geq0\bigg)$ is a MAP with respect to the filtration 
$({\cal G}_t;t\geq0) : = ({\cal F}_{A_t};t\geq0)$, i.e.
$(\xi_t; t\geq0)$ is a subordinator and \[ \E_{\xi_0,\theta_0}\left( f( \xi_{t+s} -  \xi_s,  \theta_{t+s}\right) \ | \  {\cal G}_s) \ = \  \E_{0, \theta_s}\left( f(\xi_{t},  \theta_{t}) \right). \] 
 Further, $(\mu_t; t\geq0)$ can be recovered from its Lamperti-Kiu transform through the formula
\begin{equation}  \mu_t = \theta_{\tau_t} \exp(-\xi_{\tau_t})  \label{eq:l-ch-t}\end{equation}
where $\tau_t:= \inf\{u>0  :  \int_0^u \exp(\beta \xi_s) ds>t\}$.
\end{theorem}

\subsection{Site Frequency Spectrum}

 The third part of this work is devoted to the asymptotics of the SFS of the family of general Dirichlet $\Xi$-coalescents, in the limit of large $n$. 
 \blue{Consider the infinite sites model, where it is assumed that mutations occur according to a Poisson Point Process
  of intensity $r>0$ over the coalescent tree and that each new mutation falls in a new site so that all the mutations can be observed in the generic data.} Define the rescaled SFS $F_n=(F_n(1), F_n(2),\dots)$ where 
\[\forall i\leq n-1, \ \ F_n(i)= \frac{1}{n^{{\alpha}}} \#\{\mbox{mutations carried by $i$ individuals}\}. \]
\blue{Under the infinite sites model, there is a very close relation between the SFS and the block counting process of the coalescent tree. More precisely, conditional on the coalescent, the  number of segregating mutations affecting $i$ individuals of the sample is given by a Poisson random variable with parameter $r \int_0^{\hat {\bf T}_n} \hat \mu^n_s(i) ds$, where $\hat {\bf T}_n$ denotes the time to the most recent common ancestor (height) of the coalescent tree.
}

\begin{theorem}\label{thm:cv-sfs}
Let $(\xi_t,\theta_t)$ be defined as in Theorem \ref{thm:lamperti}. 
We have
\begin{eqnarray*}
\lim_{n\to\infty} F_n & = & r  \int_0^\infty \mu_s ds \ \\
& = & r \int_0^\infty \theta_u \exp((\alpha-2)\xi_u) du,
\end{eqnarray*}
where the convergence is meant in the weak sense with respect to the  $\ell^1(\R_+)$ topology. In particular, the (rescaled) total number of mutations $ |F_n|$
is asymptotically described by the exponential functional of a subordinator \cite{BY}, i.e.,
$$
\lim_{n\to\infty} |F_n| \ = \  r \int_0^\infty  \exp((\alpha-2)\xi_u) du.
$$
\end{theorem}

Observe that a similar rescaling, in $n^{\alpha}, \alpha <1$  appears for the lower part of the spectrum (small values of $i$) of Beta-coalescents coming down from infinity \cite{BBS}, \blue{although $\alpha$ was used in a different parametrization there}.
\blue{Also note that, in most coalescent models, the rescaling order is not the same all along the vector. As an example, four different renormalizations are listed in the study the SFS of the Bolthausen-Sznitman coalescent, \cite{KSJW}.}

\subsection{Outline of the paper}
The rest of the paper is organized as follows. In Section \ref{sect-urnestimates} we use Stein's method to derive  bounds for the total variation distance between vectors obtained by throwing balls into urns and their Poisson approximations. 
These results are  used in Section \ref{sect:coupled} to prove the convergence of the generator of the multidimensional block counting process $(\mu^n_t; t\geq0)$ (defined in equation \eqref{generator-discrete})  to the generator of the limiting process (defined in \eqref{genA}). 
Section \ref{sect-def-limitingprocess} is devoted to the study of the martingale problem \eqref{MP}. Before proving the uniqueness of its solution (Theorem \ref{Mgprob}), we analyze the coagulation operator $\cC^x$ (some additional technical results can be found in Appendix B). 
In Section \ref{sect:convergence}, we prove the convergence of $(\mu^n_t; t\geq0)$ to the unique solution of the martingale problem (Theorem \ref{thm:conv-process}). 
In Section \ref{sect:selfsimilarity}, we prove that the limiting process is self-similar and we characterize its Lamperti-Kiu transform (Theorem \ref{thm:lamperti}).
We also provide an additional representation of the process using stochastic flows. 
Finally, in Section \ref{sect-SFS}, we study some asymptotics on the branch lengths which allow us to prove Theorem \ref{thm:cv-sfs}.
Appendix A contains some moment estimates on the mass partition components $p_j^{(k)}$ (defined in \eqref{masspart}) that are used in several proofs.

\section{Urn estimates}
\label{sect-urnestimates}

Let $E$ be a discrete space equipped with the usual $\sigma$-field ${\cal F}$ generated from the singletons.
Recall that the total variation distance between two measures $\nu_1,\nu_2$ on $E$ is given by
\begin{equation}\label{dTV} d_{TV}(\nu_1,\nu_2) \ = \ \sup_{A\in{\cal F}} |\nu_1(A)-\nu_2(A)| =\frac{1}{2} \sum_{x\in E} |\nu_1({x}) - \nu_2(x)|.  \end{equation}
For a random variable $X$, we denote by ${\mathfrak L}(X)$ its law.

In this section we  recall and establish  some bounds for the total variation distance between binomial variables and vectors obtained by throwing balls into urns and their Poisson approximations. Those results are mainly obtained using  Stein's method \cite{Ross}.

\subsection{Undistinguishable balls }
\label{sect-urns-1c}
We start by considering $n$ indistinguishable balls that are allocated at random to $k$ urns. 
For $i\in[k]$, let $p_i$ be the probability of being allocated to the $i$th urn. Let $X_i$ be the number of balls allocated to urn $i$ so that $X_i$ has \blue{a} binomial distribution with parameters $n$ and $p_i$ and $\sum_{i=1}^k X_i=n$. 

\begin{lemma}
\label{lemma-1urn-1color}
Let $y>0$. Let $Y$ be a Poisson distributed random variable with parameter $yn/k$. Then,
\[ d_{TV}\left({\mathfrak L} (X_1), \ {\mathfrak L}(Y)\right) \leq \min\{p_1, p_1^2 n \} + \frac{n}{k}|y - kp_1|. \]
\end{lemma}

\begin{proof}
Let $W$ be a Poisson random variable with parameter $p_1n$. Using triangular inequality, 
\begin{eqnarray*} 
d_{TV}\left({\mathfrak L}(X_1), \ {\mathfrak L}(Y)\right)
&\le& d_{TV}\left({\mathfrak L}(X_1), {\mathfrak L}(W)\right)  + d_{TV}\left({\mathfrak L}(W), {\mathfrak L}(Y)\right).
\end{eqnarray*}
For the first term in the RHS  we use the celebrated Chen-Stein inequality for the approximation of the total variation between Poisson and binomial random variables (see for example Theorem 4.6 in \cite{Ross}). 
For the second term we use an inequality for the total variation \blue{distance} between Poisson random variables with different means, that can be found in \blue{equation (5) of \cite{Roos}}. 
\end{proof}

Now we consider balls that are allocated to two different urns. 

\begin{lemma}
\label{lemma-2urns-1color}
Let $(W_1, W_2)$ be a pair of independent Poisson distributed random variables with respective parameters $p_1n$ and $p_2n$. 
Then,
\begin{eqnarray*}
 d_{TV}\left({\mathfrak L} (X_1, X_2), \ {\mathfrak L}(W_1, W_2)\right)&  \leq &   (p_1 + p_2)^2 n. 
 \end{eqnarray*}
\end{lemma}

\begin{proof}
Our argument relies on the following observation
\begin{enumerate} 
\item 
$Z:=X_1+X_2$ follows a binomial distribution of parameters $(n, p_1 + p_2)$.
\item  
Conditioned on  $Z  = z$, $(X_1, X_2)$ has a multinomial law of parameters ($z, \frac{p_1}{p_1 + p_2}, \frac{p_2}{p_1+p_2})$.
\end{enumerate}
Analogously, we can consider a pair of random variables constructed as follows. Let $W$ be a Poisson random variable of parameter $(p_1 + p_2)n$. Conditionally on $W = z$,  $(W_1, W_2)$ has the same  multinomial law of parameters ($z, \frac{p_1}{p_1 + p_2}, \frac{p_2}{p_1+p_2})$.
For $z\in\mathbb{N}_0$, denote by $B_{z}$ a binomial random variable with parameters $(z,\frac{p_1}{p_1+p_2})$. Then
\begin{eqnarray*}
&& \sum_{ i,j\ge0} \bigg |\P((X_1, X_2) = (i,j)) - \P((W_1, W_2) = (i,j))  \bigg | \\
&=&  \sum_{ i,j\ge0}  \bigg |\P(Z = i+j) \P(B_{i+j} = i) - P(W= i+j) \P(B_{i+j} = i)  \bigg | \\
&=& \sum_{z = 0}^n \sum_{i = 0}^z    \bigg |\P(Z = z)  - P(W= z)  \bigg | \P(B_{z} = i ) =  \sum_{z = 0}^n    \bigg |\P(Z = z)  - P(W= z)  \bigg |. 
\end{eqnarray*}
 So, using again Chen-Stein's inequality  (Theorem 4.6 in \cite{Ross}), we conclude that
 \[d_{TV}\left({\mathfrak L} ( X_1,  X_2), \ {\mathfrak L}( W_1,  W_2)\right) = d_{TV}\left({\mathfrak L} (Z), \ {\mathfrak L}(W)\right) \leq  \min\{(p_1 + p_2), (p_1 + p_2)^2 n \}.  \]
\end{proof}

\subsection{Balls with distinct sizes}
\label{sect-urns-ncolors}
In this section we consider an urn problem where balls are distinguishable by their sizes. We start with a general result.
\begin{lemma}\label{lemma:from-1d-tokd}
Let $\ell \in \N$ and let  $(A^{(1)},\dots, A^{(\ell+1)})$  and $(B^{(1)},\dots, B^{(\ell+1)})$ be two vectors of independent random variables \blue{defined on} the same discrete subset $E$. Then,
\[ d_{TV}\left({\mathfrak L}(A^{(1)},\dots, A^{(\ell+1)}), \ {\mathfrak L}(B^{(1)},\dots, B^{(\ell+1)})\right) \ 
\leq \ \sum_{i=1}^{\ell+1} d_{TV}(\mathfrak L(A^{(i)}),\mathfrak L(B^{(i)})).\]
\end{lemma}

\begin{proof}
Let us denote by $\nu_{A^{(i)}}$ the distribution of $A^{(i)}$ and by $\nu_{B^{(i)}}$ the distribution of $B^{(i)}$, for $i\in[\ell+1]$.  
Since the product function $f(x_1,\dots, x_{\ell+1})=x_1 \cdot ... \cdot x_{\ell+1}$ is 1-Lipshitz on $[0,1]^{\ell+1}$ for the $L^1$ norm on $\R^{\ell+1}$, the second definition of the total variation in \eqref{dTV} implies that 
\begin{eqnarray*}
&&d_{TV}\left({\mathfrak L}(A^{(1)},\dots, A^{(\ell+1)}), \ {\mathfrak L}(B^{(1)},\dots, B^{(\ell+1)})\right) \\
 &&  \ \ = \frac{1}{2} \sum_{(x_1, \dots, x_{\ell +1})\in E^{\ell + 1}} |\nu_{A^{(1)}}({x_1})\dots \nu_{A^{(\ell+1)}}({x_{\ell+1}}) - \nu_{B^{(1)}}({x_1})\dots \nu_{B^{(\ell+1)}}({x_{\ell+1}})| \\
  &&  \ \ \le \frac{1}{2} \sum_{(x_1, \dots, x_{\ell +1})\in E^{\ell + 1}}  \sum_{i=1}^{\ell+1}|\nu_{A^{(i)}}(x_i) - \nu_{B^{(i)}}({x_i})| \ = \ \sum_{i=1}^{\ell+1}  d_{TV}(\mathfrak L(A^{(i)}),\mathfrak L(B^{(i)})).
\end{eqnarray*}
\end{proof}

Fix $\ell\in\N$. We consider balls with $\ell+1$ distinct sizes. For $i\in[\ell+1]$, suppose that there are $N_i$ balls of size $i$. 
We allocate at random these balls into $k$ urns, with probabilities $(p_1, \dots, p_k)$. 
Let $(X^{(1)}_j, \dots, X_j^{(\ell+1)})$ be the number of balls of each size that are allocated to urn $j$. 
The variables $(X^{(1)}_j, \dots, X_j^{(\ell+1)})$ are independent binomial random variables  with respective parameters  $(N_i,p_j)$.  
The following result is a direct consequence of Lemmas \ref{lemma-1urn-1color} and \ref{lemma:from-1d-tokd}.
\begin{corollary}\label{lemma-1urn-dcolors}
Let $(X^{(1)},\dots, X^{(\ell+1)})$ be independent binomial random variables   with respective parameters  $(N_1,p_1),\dots, (N_{\ell+1},p_1)$. Let $y_1>0$ and let  $(Y^{(1)},\dots, Y^{(\ell+1)})$ be independent Poisson random variables  with respective parameters  $y_1N_1/k,\dots, y_1N_{\ell+1}/k$. 
Let $\bar N=\sup_i N_i$.
We have 
\begin{eqnarray*}
 d_{TV}\left({\mathfrak L}(X^{(1)},\dots, X^{(\ell+1)}), \ {\mathfrak L}(Y^{(1)},\dots, Y^{(\ell+1)})\right) &\leq &  (\ell+1) \big(\min\{p_1, p_1^2 \bar N \} + \frac{\bar N}{k}|y_1 - kp_1|\big). 
 \end{eqnarray*}
\end{corollary}

For each $ i \in [\ell+1]$, let $\vec X^{(i)} := (X_1^{(i)}, X_2^{(i)})$ 
denote the number of balls of size $i$ that are allocated to urn $1$ and $2$ respectively.
The variable $X_1^{(i)}$ (resp. $X_2^{(i)}$) follows a binomial distribution of parameters $(N_i, p_1)$ (resp., $(N_i, p_2)$).
The following result is a direct consequence of Lemma \ref{lemma:from-1d-tokd} combined first with Lemma \ref{lemma-2urns-1color} and second with the Chen-Stein inequality.
\begin{corollary}\label{lemma-2urns-dcolors}
 Let  $(\vec W^{(1)},\dots, \vec W^{(\ell+1)})$ be independent random vectors such that $\vec W^{(i)} = (W_1^{(i)}, W_2^{(i)})$ is a couple of independent Poisson random variables with respective parameters $p_1N_i$, $p_2N_i$. Then,
\begin{eqnarray*}
 d_{TV}\left({\mathfrak L} (\vec X^{(1)}, \dots, \vec X^{(\ell+1)}), {\mathfrak L}(\vec W^{(1)},\dots, \vec W^{(\ell+1)})\right)  \leq  (\ell+1) (p_1 + p_2)^2 \bar N, 
 \end{eqnarray*}
 and, for $j \in \{1, 2\}$,
 \begin{eqnarray*}  d_{TV}\left({\mathfrak L} ( X_j^{(1)}, \dots,  X_j^{(\ell+1)}), {\mathfrak L}( W_j^{(1)},\dots,  W_j^{(\ell+1)})\right)  \leq  (\ell+1) p_j^2 \bar N, 
 \end{eqnarray*}
 where $\bar N=\sup_i\{N_i\}$.
\end{corollary}

\subsection{Coagulation operators defined from urn  problems}
Again, we fix $\ell\in\N$ and consider balls with $\ell+1$ distinct sizes. Let $\vec N = (N_1, \dots, N_{\ell + 1})$ where $N_i$ denotes the number of balls of size $i$. 
We allocate at random these balls into $k$ urns, such that the probabilities of falling in \blue{different urns} are given by $\vec p = (p_1, \dots, p_k)$. 
We define a random coagulation of $\vec N$ by considering that balls that are assigned to the same urn are merged into one ball whose size is the sum of all of them, i.e., such that 
\[ \forall m \in \N, \ \mathbf{C}^{\vec p}(\vec N)(m)  \ = \ \#\{j\le k : \mbox{ sum of the sizes of the balls falling in box $j$ is $m$} \}. \]

\begin{lemma}
\label{lemma-combin}
Fix a probability vector $\vec p=(p_1,\dots,p_k) \in [0, 1]^k$ and consider the random coagulation associated to $\vec p$. 
For every $\vec N \in \N^{\ell + 1}$, 
\begin{equation}\label{Nm-Cm}
\forall m \in \N, \ \ \big|N_m  - \mathbf{C}^{\vec p}(\vec N)(m)\big | \le  2 (|\vec N|- |\mathbf{C}^{\vec p }(\vec N)|),
\end{equation}
almost surely, and
 \begin{equation}\label{N-C}
 \E(|\vec N|- |\mathbf{C}^{\vec p }(\vec N)|) \ \le \ |\vec N|^2 \sum_{j=1}^k p_j^2.  
 \end{equation}
\end{lemma}

   \begin{proof}
Let us first prove \eqref{Nm-Cm}.  The difference $N_m  - \mathbf{C}^{\vec p}(\vec N)(m)$ is the net gain or loss of balls of size $m$ in the coagulation operation. This can be computed by taking the difference of the following two quantities
 \begin{enumerate} 
\item[(a)]  The number of balls of size $m$ falling in an urn where another ball is assigned (regardless of its size).
 \item[(b)] The number of urns with more than two balls and whose sizes add up to $m$.
 \end{enumerate}
 Suppose first that $N_m  > \mathbf{C}^{\vec p}(\vec N)(m)$.
Thus, $N_m  - \mathbf{C}^{\vec p}(\vec N)(m)$ is smaller than (a) alone. Finally, (a) is smaller than the twice the total number of balls that are lost. 
As an illustrative example, consider the case when the $N_m$ balls are assigned in pairs to $N_m/2$ different urns, coagulating into $N_m/2$ balls of size $2m$. Then the number of balls of size $m$ that are lost is $N_m$ and the total number of balls that are lost is $N_m/2$.
 
 Now suppose that $N_m  < \mathbf{C}^{\vec p}(\vec N)(m)$. In this case $\mathbf{C}^{\vec p}(\vec N)(m)-N_m$ is less than (b). In turn, this is less than the number of urns containing at least two balls, which is less than the total number of balls that are lost. This completes the proof of \eqref{Nm-Cm}.

 Now, let us turn to the proof of \eqref{N-C}. Since for $x\in(0,1)$, we have $\log(1-x) \le - x $ and $e^{-x}\le1-x+x^2$, then
\begin{eqnarray*}
1-(1-p)^n =1- e^{n\log(1-p)} \ge  1-e^{-np } \ge np-(np)^2.
\end{eqnarray*}
Recall that $|\mathbf{C}^{\vec p }(\vec N)|$ is the number of non-empty boxes when assigning $|\vec N|$ balls to $k$ urns with probabilities $(p_i, \dots, p_k)$. Using the previous inequality,
\begin{eqnarray*}
 \E(|\vec N|- |\mathbf{C}^{\vec p }(\vec N)|)  &=&  |\vec N| -  \sum_{j=1}^k (1 - (1-p_j)^{|\vec N|})\\
&\le&|\vec N| - \sum_{j=1}^k \bigg( |\vec N|  p_j - |\vec N|^2 p_j^2\bigg).
\end{eqnarray*}
Observing that $\sum_{j=1}^k p_j = 1$ yields the desired result.
  \end{proof}

\section{Sequence of urns. Convergence of the generator}
\label{sect:coupled}

Recall from Section \ref{mainres} the continuous generator $\mathcal A$ defined in \eqref{genA} and the discrete generator $\mathcal A_n$ defined in \eqref{generator-discrete}.
We use the urn estimates obtained in Section \ref{sect-urnestimates} to prove the following result.
\begin{proposition}\label{prop:cv-law}
Let $f \in \mathcal{T} $. Consider a sequence $(\bz_n;n\in\N)$ with $\bz_n\in {\cal Z}_n$ such that $\bz_n \to \bz $ in $\ell^1(\R_+)$. Then  
\[  \mathcal{A}_n f( \bz_n )   \ \to \  \mathcal{A} f( \bz ).\]
\end{proposition}

\label{sect-momentestimates}
\subsection{First moment estimate}
Recall the notations ${\cal C}^x$ from \eqref{defCx} and $\Lambda^{k,n}$ from \eqref{defLkn}. The main objective of this section is a careful justification of the approximation of $\Lambda^{k,n}({\bf z}_n)(\ell) $ by $ {\cal C}^{k/n}(\bz_n)(\ell)$ , as suggested in the heuristics provided  in Section \ref{heuristics}, together with a ``rate of convergence'' that will be needed to prove 
Proposition \ref{prop:cv-law}.
\begin{lemma}\label{lemE-C}
For every $\ell\in\N$ and ${\bf z}_n \in{\cal Z}_n$, we have 
\begin{eqnarray*}
\bigg|\E(\Lambda^{k,n}(\bz_n)(\ell)) -\cC^{{k}/{n}}(\bz_n)(\ell)   \bigg| &\le& 2  (\ell+1) \frac{k}{n} \E\big( p^{(k)}_1 +  \frac{n}{k}|\Gamma - kp_1^{(k)}| \big) 
\end{eqnarray*}
and
\begin{eqnarray*}
\bigg|\E(|\Lambda^{k,n}(\bz_n)| - |\cC^{{k}/{n}}(\bz_n)|  \bigg| &\le& 2  \frac{k}{n} \E\big( p^{(k)}_1 +  \frac{n}{k}|\Gamma - kp_1^{(k)}| \big), 
\end{eqnarray*}
 where we recall that $p_1^{(k)}= w_1/\sum_{i=1}^k w_i$ and $\Gamma = w_1/\E(w_1)$ so that the random variables are coupled through the same $w_1$.
\end{lemma}

\begin{proof}
We start by proving the first inequality. Fix $\ell \in \N$.
 For $i \in [\ell]$, define $N_{i} = nz_n(i)$, and set $N_{\ell+1} =  n(|{\bf z}_n| - \sum_{i=1}^\ell z_n(i))$. In terms of the urn problem of Section \ref{sect-urnestimates}, $N_{i}$ is the number of balls of size $i$, for $i\in[\ell]$, and
$N_{\ell+1}$ is the number of balls of size strictly larger than $\ell$.
Let us now consider a partition  $c \in \phi^{-1}(\ell)$, i.e., a vector $c\in\ell^1(\R_+)$ such that $\sum i c(i)=\ell$. An urn containing $c(i)$ balls of size $i$ for each $i \in [\ell]$ corresponds to the formation of a new block of size $\ell$. 
Let $B(c)$ denote the number of urns containing balls given by the partition $c$ (to ease the notation we do not indicate the dependence on $\bz_n$ and $k$). 
We have 
$$\Lambda^{k,n}(\bz_n)(\ell) \ = \ \frac{1}{n} \sum_{c\in \phi^{-1}(\ell) } B(c).
$$
Mirroring the notation of Section \ref{sect-urns-ncolors}, consider a vector of r.v.'s $(X^{(1)},\dots, X^{(\ell+1)})$ such that, conditional on $p_1^{(k)}$, the entries are independent 
and  $X^{(i)}$ is distributed as a binomial r.v. with parameters $(N_i,p_1^{(k)})$.
By exchangeability of the boxes, we have
$$
\E\left(\Lambda^{k,n}(\bz_n)(\ell)\right)  = \frac{k}{n} \E\left(\sum_{c\in \phi^{-1}(\ell)} b_{1}(c)\right),
$$
where $b_{1}(c)$ is the indicator that $X^{(i)}=c(i)$ for every $i\in[\ell]$ and that $X^{(\ell+1)}=0$.

Similarly, let $(Y^{(1)},\dots, Y^{(\ell+1)})$ such that, conditional on $\Gamma$, the entries are independent 
and  $Y^{(i)}$ is distributed as a Poisson r.v. with parameter $\Gamma N_i/k$. Let  $d_{1}(c)$ be the indicator that $Y^{(i)}=c(i)$ for every $i\in[\ell]$ and that $Y^{(\ell+1)}=0$.
A direct computation shows that 
$$
\cC^{{k}/{n}}\left(\bz_n\right)(\ell)  \ = \  \frac{k}{n}  \E\left(\sum_{c\in \phi^{-1}(\ell)} d_{1}(c)\right).
$$
It follows that 
\begin{eqnarray*}
&&\big| \E\left(\Lambda^{k,n}(\bz_n)(\ell)\right)  - \cC^{{k}/{n}}\left(\bz_n\right)(\ell) | =  \frac{k}{n}  \left|\sum_{c\in \phi^{-1}(\ell)}\left( \E(b_{1}(c)) - \E(d_1(c)\right)\right | \\
&\leq & 2  \frac{k}{n}  d_{TV} ({\mathfrak L}(X^{(1)},\dots, X^{(\ell+1)}), {\mathfrak L} (Y^{(1)},\dots, Y^{(\ell+1)})).
\end{eqnarray*}
The result follows by a direct application of Corollary \ref{lemma-1urn-dcolors} after conditioning on $p_1^{(k)}$ and $\Gamma$. 

\medskip

We now prove the second inequality.
Mirroring the notation of Section \ref{sect-urns-1c}, we consider the random variable $X_1$ such that conditional on $p_1^{(k)}$,   $X_1$ is distributed as a binomial r.v. with parameters $(n|\bz_n|,p_1^{(k)})$. Let $b_{1}$ be the indicator that box 1 contains at least one ball.
Similarly, let $Y_1$ be the random variable such that, conditional on $\Gamma$, $Y_1$ is distributed as a Poisson r.v. with parameter $\Gamma n|\bz_n| /k$. Let $d_1$ be the indicator that $Y_1 \ge 1$. We have
$$
\E\left(|\Lambda^{k,n}(\bz_n)|\right)  = \frac{k}{n} \E\left( b_{1}\right) \textrm{ and } |\cC^{{k}/{n}}\left(\bz_n\right)| = \frac{k}{n}  \E\left( d_{1}\right)
$$
and the result then follows from Lemma \ref{lemma-1urn-1color}.
\end{proof}

\subsection{Second moment estimate}
The aim of this section is to bound the variance of the operator $\Lambda^{k,n}({\bz})(\ell)$. 

\begin{lemma}\label{lemma-secondmoment1}
For every $\ell\in\N$, there exists a constant $C>0$ such that for every ${\bf z}_n \in{\cal Z}_n$, we have
\begin{eqnarray*} \Var\left(\Lambda^{k,n}(\bz_n)(\ell) \right)  & \leq & C ( \frac{1}{n} + \max\{1, \frac{k}{n}\}  \ h(k) )
\end{eqnarray*}
where  $h$ is a function of $k$,  with no dependence in $n,z_n$ which goes to $0$ as $k\to\infty$.
Further, 
the same property holds for $\Var\left(|\Lambda^{k,n}(\bz_n)| \right)$.
\end{lemma}
\begin{proof}
We start by proving the inequality for $\Var\left(|\Lambda^{k,n}(\bz_n)|\right)$.
Fix $\ell \in \N$. In the following, we write $N_i = n z_{n}(i)$.
We consider the vector
 $(\vec X^{(1)},\dots,\vec X^{(\ell+1)})$ where the entries
$\vec X^{(i)} =(X^{(i)}_1, X^{(i)}_2)$ are, conditional on  $(p_{1}^{(k)}, p_{2}^{(k)})$,
 independent random vectors such that $X^{(i)}_1+ X^{(i)}_2$ is binomial with parameters $(N_i, p_1^{(k)}+ p_2^{(k)})$ and, conditional on  \{$X^{(i)}_1+ X^{(i)}_2=z\}$, $\vec X^{(i)}$ is  multinomial with parameters $(z, \frac{p_1^{(k)}}{p_1^{(k)}+ p_2^{(k)}}, \frac{p_2^{(k)}}{p_1^{(k)}+ p_2^{(k)}})$.
Analogously to Lemma \ref{lemE-C}, we define $b_{j}(c)$ (with $j=1,2$) as the indicator that $X_j^{(i)}=c(i)$ for $i\in[\ell]$ and that $X_j^{(\ell+1)}=0$.
Adapting the notations in the proof of Lemma \ref{lemE-C} and using again exchangeability between urns,  
\begin{eqnarray}
 \Var\left(\Lambda^{k,n}(\bz_n)(\ell) \right) 
& = & \frac{k(k-1)}{  n^2} \sum_{c_1, c_2\in \phi^{-1}(\ell) } \Cov\left(b_{1}(c_1) , b_{2}(c_2)  \right)  \nonumber\\
& + & \frac{k}{n^2} \displaystyle\sum_{c_1\in \phi^{-1}(\ell)} \Cov\left(b_{1}(c_1), b_{1}(c_1)   \right). 
\label{eq-cov0}
\end{eqnarray}

{\bf Step 1.} We start by considering the first term in \eqref{eq-cov0}. 
Define  $(\vec W^{(1)},\dots,\vec W^{(\ell+1)})$ such that the entries
$\vec W^{(i)} =(W^{(i)}_1, W^{(i)}_2)$ are, conditional on  $(p_{1}^{(k)}, p_{2}^{(k)})$,
 independent random vectors such that $W^{(i)}_1$ and $W^{(i)}_2$ are independent Poisson variables with parameters $N_i p_1^{(k)}$ and $N_ip_2^{(k)}$. 
 We define $d_{j}(c)$ (with $j=1,2$) as the indicator that $W_j^{(i)}=c(i)$ for $i\in[\ell]$ and that $W_j^{(\ell+1)}=0$. 
We have
\begin{eqnarray}
|\Cov\left(b_{1}(c_1) , b_{2}(c_2)  \right)| 
& \le& \bigg| \E\left(b_{1}(c_1) b_{2}(c_2)  \right) -\E\left(d_{1}(c_1) d_{2}(c_2)  \right) \bigg| \nonumber\\
&+&  \bigg| \E\left(d_{1}(c_1) d_{2}(c_2)\right)-\E\left(d_{1}(c_1)\right) \E\left( d_{2}(c_2) \right)  \bigg| \nonumber\\
&+&  \bigg|\E\left(d_{1}(c_1)\right) \E\left( d_{2}(c_2)\right)- \E\left(b_{1}(c_1)\right) \E\left( b_{2}(c_2) \right) \bigg|\nonumber \\
& \le & \bigg| \E\left(b_{1}(c_1), b_{2}(c_2)  \right) -\E\left(d_{1}(c_1) d_{2}(c_2)  \right) \bigg| \nonumber \\
&+&|\Cov(d_{1}(c_1), d_{2}(c_2))| 
+ 
\sum_{j=1}^2  | \E\left(b_{j}(c_j)\right) - \E\left(d_{j}(c_j)\right)  |,
\label{ineq-covariance}
\end{eqnarray}
where, for the last term of the second inequality, we used the fact that the product function $f(x_1, x_2) = x_1x_2$ is 1-Lipshitz on $[0,1]^2$ in the $L^1$ norm of $\R^2$.
We now consider each of the terms in the RHS of \eqref{ineq-covariance} separately.

We start with the third term.
The  result follows from the second item of Corollary \ref{lemma-2urns-dcolors} and point $(ii)$ of Proposition \ref{prop:moments} (in Appendix A). More precisely, 
as in the proof Lemma \ref{lemE-C}, 
\begin{eqnarray}
&&\frac{k(k-1)}{n^2}  \sum_{c_1 , c_2\in \phi^{-1}(\ell) } \sum_{j=1}^2  | \E\left(b_{j}(c_j)\right) - \E\left(d_{j}(c_j)\right)  |\nonumber\\ &\le& 4\frac{k(k-1)}{n^2}  d_{TV} ({\mathfrak L}(X_1^{(1)},\dots, X_1^{(\ell+1)}), {\mathfrak L} (W_1^{(1)},\dots, W_1^{(\ell+1)})) \nonumber \\ &\le & \frac4n    (\ell+1)  \E((kp^{(k)}_j)^2)\nonumber \\
&\le& C \frac1n,
\label{ineq2termcov}
\end{eqnarray}
where $C$ is a positive constant.

For the first term, we observe that 
\begin{eqnarray*}
 &&\sum_{c_1 , c_2\in \phi^{-1}(\ell) }\bigg|    \E\left(b_{1}(c_1) b_{2}(c_2)  \right) -\E\left(d_{1}(c_1) d_{2}(c_2)  \right) \bigg|  \\
&\leq & 2  d_{TV} ({\mathfrak L}(\vec X^{(1)},\dots, \vec X^{(\ell+1)}), {\mathfrak L} (\vec W^{(1)},\dots,\vec  W^{(\ell+1)})) .
\end{eqnarray*}
We can then apply the first item of Corollary \ref{lemma-2urns-dcolors} to obtain the following bound
\begin{eqnarray}
&&\frac{k(k-1)}{n^2}  \sum_{c_1 , c_2\in \phi^{-1}(\ell) } \bigg| \E\left(b_{1}(c_1) b_{2}(c_2)  \right) -\E\left(d_{1}(c_1) d_{2}(c_2)  \right) \bigg|  \nonumber \\ &\le &\frac 2n (\ell+1)  \E(k^2(p^{(k)}_1+ p^{(k)}_2)^2)\nonumber \\
&\le & \frac2n (\ell+1) 4\E((kp_1^{(k)})^2)  \le \ C  \frac1n,
\label{ineq1termcov}
\end{eqnarray}
where $C$ is a positive constant and we used the fact that $\E((kp_1^{(k)})^2)$ is finite by point $(ii)$ of Proposition \ref{prop:moments} (in Appendix A).

Finally, we deal with the second term of \eqref{ineq-covariance},
\begin{eqnarray}
|\Cov(d_{1}(c_1), d_{2}(c_2))| & =& \bigg |\E\bigg(\Cov(d_{1}(c_1), d_{2}(c_2)) \ \vert \  p^{(k)}_1, p^{(k)}_2)\bigg)  \nonumber \\
&+& \ \ \Cov\bigg(\E ( d_{1}(c_1)\ \vert  \ p^{(k)}_1, p^{(k)}_2), \ \E(d_{2}(c_2) \ \vert \   p^{(k)}_1, p^{(k)}_2) \bigg)\bigg | \nonumber \\
 & =& |\Cov\bigg(\E ( d_{1}(c_1)\vert  p^{(k)}_1, p^{(k)}_2), \E(d_{2}(c_2) \vert  p^{(k)}_1, p^{(k)}_2)\bigg)| \nonumber \\
 &=& |\Cov(g_{v, c_1}(k p_1^{(k)}), g_{v, c_2}(k p_2^{(k)})) |\nonumber \\
 &\le& \max\{\frac{n^2}{k^2}, \frac{n}{k}\} h(k).
\label{ineq3termcov}
\end{eqnarray}
where $g_{v, c}$ and $h$ are defined as in Lemma \ref{lemma:g_v,c} (in Appendix A) with  $v= \frac{n}{k}(z_n(1), \dots, z_n(\ell), |z_n|)$, where $v(\ell+1) = \frac{n}{k}|z_n| \le \frac{n}{k}$.
The inequality follows from that lemma.

Combining the three inequalities \eqref{ineq2termcov}, \eqref{ineq1termcov} and \eqref{ineq3termcov}, there exists a constant $C$ such that the first term in the RHS of \eqref{eq-cov0} can be bounded by
\begin{eqnarray*}
&& \frac{k(k-1)}{  n^2} \sum_{c_1 \in\phi^{-1}(\ell_1), c_2\in \phi^{-1}(\ell_2) } \Cov\left(b_{1}(c_1) , b_{2}(c_2)  \right) \\
 &\le & C\frac1n  \ + \ h(k) \max\{1, \frac{k}{n}\}  ,
\end{eqnarray*}
where $h(k) \to 0$ as $k\to\infty$.

{\bf Step 2.} Now we consider the second term in  \eqref{eq-cov0},
\begin{eqnarray*}
\Cov\left(b_{1}(c_1), b_{1}(c_1)  \right) = \Var(b_{1}(c_1)) &=&  \P(b_{1}(c_1) = 1)(1-\P(b_{1}(c_1) = 1) ) \\
&\le& \P(b_{1}(c_1) = 1) \le \E(p_1^{(k)}),
\end{eqnarray*}
where we used the fact that $\P(b_{1}(c_1) = 1) $ is bounded from above by the probability of the event of having at least one ball in the first box.
So the second term in  \eqref{eq-cov0} can be bounded by
\begin{eqnarray*}
 \frac{k}{n^2} \sum_{c_1\in \phi^{-1}(\ell) } \Cov\left(b_{1}(c_1), b_{1}(c_1)   \right) &\le& \frac1{n^2}  \E(kp_1^{(k)})|\phi^{-1}(\ell)|.
\end{eqnarray*}
This completes the proof of  the first inequality of Lemma \ref{lemma-secondmoment1}. 

\medskip

We now prove the inequality for $\Var\left(|\Lambda^{k,n}(\bz_n)| \right)$. We write 
$$|\Lambda^{k,n}(\bz_n)| \ = \ \frac{1}{n} (k - B_0),
$$
where $B_0$ is the number of empty boxes. The proof follows the same steps as the proof of the first inequality. If $b_{i, 0}$ is the indicator that box $i$ is empty
\begin{eqnarray*}
 \Var\left(|\Lambda^{k,n}(\bz_n)| \right) 
& = & \frac{k(k-1)}{n^2}  \Cov\left(b_{1, 0} , b_{2, 0}  \right) 
 +  \frac{k}{n^2} \Var\left(b_{1, 0}, b_{1, 0}   \right). 
\end{eqnarray*}
In the first step, analogously to \eqref{ineq-covariance}, we can write
\begin{eqnarray*}
|\Cov\left(b_{1, 0} , b_{2, 0}  \right)| 
\le | \E\left(b_{1, 0}, b_{2 , 0}  \right) -\E\left(d_{1, 0}, d_{2, 0} \right) |+|\Cov(d_{1, 0}, d_{2, 0})| 
+   \sum_{j=1}^2  | \E\left(b_{j, 0}\right) - \E\left(d_{j, 0}\right)  |
\end{eqnarray*}
\blue{where $d_{1,0}, i = \{1,2\}$ is the indicator that a Poisson r.v. with parameter $n |\bz_n|p_i^{(k)}$ is equal to 0. Conditioning on  $(p_{1}^{(k)}, p_{2}^{(k)})$,  $d_{1,0}$ and  $d_{2,0}$ are independent. 
 For the first term we use Lemma \ref{lemma-2urns-1color}.
For the second term, we use a similar bound to equation \eqref{ineq3termcov}, where $g_{v, c_i}, \ i=1,2$  is replaced by $g_{v, c_0}$ defined by $g_{v, c_0}(x):= \exp(-v(\ell + 1)x)$ that has Lipshitz constant $v(\ell + 1) = n/k$ and the inequality follows from Lemma \ref{lemma-lipschitz} instead of Lemma \ref{lemma:g_v,c}.
For the third term we use the Chen-Stein inequality.}
For the second step, the proof is analogous to the proof of the first inequality.
\end{proof}

\subsection{Convergence of the generators}
The previous sections provide the main ingredients to prove Proposition \ref{prop:cv-law}. Before writing this proof, we still need one preliminary result.
\begin{lemma}
\label{lemma-Rna}
Fix $f\in{\cal T}$. For any $A>0$ and $n\in\N$, define
\blue{
\[ 
{\cal R}_{n,A}   \ :=  \sup_{\bz_n \in {\bf z}_n } \left \{ \sum_{k= An+1}^{\infty} \frac{R(k)}{n^{1-\alpha}}\bigg| \E\left( \Lambda^{k,n}f(\bz_n ) \right) - f(\bz_n) \bigg| \right \} ,
\] }
 we have $\lim_{A\to\infty} \lim_{n\to\infty} {\cal R}_{n,A} \ = \ 0.$
\end{lemma}
\begin{proof}
Since $f\in {\cal T}$, there exists a Lipshitz function $F$   and ${\vec \lambda}\in [0,1]^{K}$ such that $f = F\circ \psi_{\vec {\lambda}}$. As a consequence, there exists $C>0$ (that only depends on the choice of $F$), such that
\begin{eqnarray*}\label{eq:sum}
\forall \bz_n\in{\cal Z}_n, \ \ \bigg| \Lambda^{k,n} f(\bz_n) - f(\bz_n) \bigg|     \le 
C \sum_{i=1}^K  \bigg| \E\left(\Lambda^{k,n}\psi_{\lambda_i}(\bz_n)\right)  -  \psi_{\lambda_i}(\bz_n)  \bigg| .
\end{eqnarray*}
In addition, if $\lambda <1$
\begin{eqnarray*} 
\forall \bz_n\in{\cal Z}_n, \ \  \bigg| \E \left( \Lambda^{k,n}\psi_{\lambda}(\bz_n)\right)  -  \psi_{\lambda}(\bz_n)  \bigg| &\le&  \sum_{\ell = 1}^\infty \lambda^\ell |\E \left( \Lambda^{k,n}(\bz_n)(\ell) - \bz_n(\ell) \right)| \\
&\le&  \E\left(|\bz_n| - |\Lambda^{k,n}(\bz_n)|  \right) 2 \sum_{\ell = 1}^\infty \lambda^\ell  , 
\end{eqnarray*}
where in the last line we used  the first part of Lemma \ref{lemma-combin}. 
If $\lambda = 1$, 
\begin{eqnarray*} 
\forall \bz_n\in{\cal Z}_n, \ \  \bigg| \E \left( \Lambda^{k,n}\psi_{\lambda}(\bz_n)\right)  -  \psi_{\lambda}(\bz_n)  \bigg| =  \E\left(|\bz_n| - |\Lambda^{k,n}(\bz_n)|  \right).
\end{eqnarray*}
Observe that $n( |\bz_n| - |\Lambda^{k,n}(\bz_n)| )\ge0$ is the number of blocks that are lost in the coalescence event. Thus, 
we can apply \eqref{N-C} (with $\vec N = n \bz_n$), and set the constant $C' = 2\lambda C/(1-\lambda)$ if $\lambda < 1$ and $C$ if $\lambda = 1$ so that
\begin{eqnarray*}
\sum_{k= An+1}^{\infty} \frac{R(k)}{n^{1-\alpha}}\bigg| \E\left( \Lambda^{k,n}f(\bz_n ) \right) - f(\bz_n) \bigg| 
&\leq  & C'  \ \sum_{k= An+1}^{\infty}\frac{R(k)}{n^{2-\alpha}} \left( n^2 |\bz_n|^2 \E(\sum_{j=1}^k (p_j^{(k)})^2) \right) \nonumber \\
&\leq  &  \frac{C'}n\sum_{k= An+1}^{\infty}\frac{R(k)}{k n^{-\alpha-1}} \E((kp_1^{(k)})^2) .
\label{eq:rna}
\end{eqnarray*}
Notice that in the last line we used the exchangeability of the vector $(p_1^{(k)}, \dots,p_k^{(k)})$ and that the bound does not depend on $\bz_n $.
By Assumptions \ref{def:finite-xi}, $\frac{1}{n} \sum_{k > An } \frac{R(k)}{kn^{-\alpha-1}}\to \int_A^{\infty}\rho\frac{dx}{x^{1+\alpha}}$, 
and by point $(ii)$ of Proposition \ref{prop:moments} (in Appendix A),  $\E((kp_1^{(k)})^2) \to \E(w_1^2)/\E(w_1)^2$, which yields the desired result.
\end{proof}
As a consequence, we get a result that will be useful later on to obtain dominated convergence.

\begin{corollary}\label{cor:ui}
For every $f\in{\cal T}$, $\sup_{n\in\N, {\bz_n\in{\cal Z}_n}}\mathcal{A}_n f( \bz)<\infty$.
\end{corollary}
\begin{proof}
First,
$$
\bigg| \mathcal{A}_n f( \bz_n ) \bigg|  \ \leq  \   \sum_{k= 1}^{n} \frac{R(k)}{n^{1-\alpha}}\bigg| \E(\Lambda^{k,n} f(\bz_n)) - f(\bz_n) \bigg| + \ {\cal R}_{n,1}.
$$
The second term on the RHS is bounded by Lemma \ref{lemma-Rna}. Finally,
$$
\sum_{k= 1}^{n} \frac{R(k)}{n^{1-\alpha}}\bigg| \E(\Lambda^{k,n} f(\bz_n)) - f(\bz_n) \bigg| \leq 2 ||f||_\infty \sum_{k=1}^n \frac{R(k)}{n^{1-\alpha}}.
$$
Since $R(k)\sim \rho k^{-\alpha}$, the RHS is converging to $\rho \int_0^1 \frac{dx}{x^\alpha}<\infty$. This completes the proof.
\end{proof}

\begin{proof}[Proof of Proposition \ref{prop:cv-law}]
Let $f\in {\cal T}$. We divide $\mathcal{A}_n f( \bz_n )$ in two parts.
$$\mathcal{A}_n f( \bz_n )= \sum_{k= 1}^{An} \frac{R(k)}{n^{1-\alpha}}\bigg(  \E(\Lambda^{k,n}f(\bz_n))   - f(\bz_n) \bigg) + \ \sum_{k= An+1}^{\infty} \frac{R(k)}{n^{1-\alpha}}\bigg(  \E(\Lambda^{k,n}f(\bz_n))   - f(\bz_n) \bigg).
$$
We proceed  by taking successive limits, first when $n\to \infty$ and then when $A\to \infty$.
By Lemma \ref{lemma-Rna}, the second term in the RHS vanishes and it remains to show that
$$
\lim_{A\to\infty} \lim_{n\to\infty} \sum_{k= 1}^{A n} \frac{R(k)}{n^{1-\alpha}}\bigg(  \E(\Lambda^{k,n}f(\bz_n) )  - f(\bz_n) \bigg) =  \int_0^{+\infty} \bigg( {\cal C}^x f({\bf z}) - f({\bf z}) \bigg)\rho x^{-\alpha}dx. 
$$  
Next,
\begin{eqnarray}
  \sum_{k= 1}^{An} \frac{R(k)}{n^{1-\alpha}} \bigg( \E(\Lambda^{k,n}f(\bz_n))   - f(\bz_n) \bigg) &=& \sum_{k = 1}^{An} \frac{R(k)}{n^{1-\alpha}} \left( \mathcal{C}^{{k}/{n}} f(\bz_n)- f(\bz_n)\right)\nonumber\\
&+&  \sum_{k = 1}^{An} \frac{R(k)}{n^{1-\alpha}} \left(f( \E( \Lambda^{k,n} {\bf z}_n ))  -\mathcal{C}^{{k}/{n}} f(\bz_n) \right) \nonumber \\
&+&\sum_{k = 1}^{An} \frac{R(k)}{n^{1-\alpha}} \bigg(\E(\Lambda^{k,n}f(\bz_n))-  f( \E( \Lambda^{k,n} {\bf z}_n ) )   \bigg),
\label{eq:generator2-3terms}
\end{eqnarray}
where the expectation is taken coordinatewise.
The rest of the proof will be decomposed into three steps. In the first one, we will show that the first term converges to ${\cal A} f(\bz)$. In the second and third ones, we will show
that the second and third terms on the RHS vanish. To do so, we will use the first and second moment estimates derived in this section.

{\bf Step 1.} We have

\begin{eqnarray*}
\frac{1}{n}\sum_{k = 1}^{An} \frac{R(k)}{n^{-\alpha}} \left(  \mathcal{C}^{{k}/{n}} f(\bz_n) - f(\bz_n)\right) &=&
\frac{1}{n}\sum_{k = 1}^{An} \frac{R(k)}{n^{-\alpha}} \left(  \mathcal{C}^{{k}/{n}} f(\bz) - f(\bz)\right)  \\
&+& \frac{1}{n}\sum_{k = 1}^{An} \frac{R(k)}{n^{-\alpha}} \left(  \mathcal{C}^{{k}/{n}} f(\bz_n) - f(\bz_n) - \mathcal{C}^{{k}/{n}} f(\bz) + f(\bz)\right). 
\end{eqnarray*}
For the first term, 
we first note that ${\cal C}^{x}(f)({\bz})$ is continuous in $x$. (This can be shown by a standard domination argument).  This implies that 
\[ \frac{1}{n}\sum_{k = 1}^{An} \frac{R(k)}{n^{-\alpha}} \left(  \mathcal{C}^{{k}/{n}} f(\bz) - f(\bz)\right) \ 
\to \ \int_0^A  \bigg( \mathcal{C}^{x}f(\bz)  - f(\bz) \bigg) \rho x^{-\alpha} dx. \]
We now prove that the second term converges to 0. Since $f\in {\cal T}$, there exists a Lipshitz function $F$ and ${\vec \lambda}\in [0,1]^{K}$ such that $f = F\circ \psi_{\vec {\lambda}}$, so there exists $C>0$ such that
\begin{eqnarray*}
 \vert \ \mathcal{C}^{{k}/{n}} f(\bz_n) - \mathcal{C}^{{k}/{n}} f(\bz) \vert  \le C \vert (\psi_{\vec \lambda}( \mathcal{C}^{{k}/{n}}(\bz_n)) - \psi_{\vec \lambda}( \mathcal{C}^{{k}/{n}}(\bz)) \vert. 
\end{eqnarray*}
It will be shown in Proposition \ref{prop1} that for $\lambda \in [0, 1/4)$,
\begin{eqnarray*}
\psi_\lambda( \mathcal{C}^{{k}/{n}}(\bz)) & = & 
\frac{k}{n} \E\left( \exp(-(|\bz| - \psi_\lambda(\bz))\frac{\Gamma n}{k}) - \exp(-|\bz|\frac{\Gamma n}{k}) \right), 
\end{eqnarray*}
where $|\bz| - \psi_\lambda(\bz)\ge 0$. Since the exponential function is Lipschitz on $(-\infty, 0)$, there exists a constant $B$ such that
\begin{eqnarray*}
\bigg |\psi_\lambda( \mathcal{C}^{{k}/{n}}(\bz_n)) - \psi_\lambda( \mathcal{C}^{{k}/{n}}(\bz)) \bigg | & \le & B \E(\Gamma) \left( \left||\bz| - |\bz_n|\right|  + \left| \psi_\lambda(\bz)  - \psi_\lambda(\bz_n)\right | \right).
\end{eqnarray*}
 This bound is independent of $k$ and goes to 0 as $n\to \infty$, which completes the proof.

{\bf Step 2.} 
We prove that the absolute value of the second term in \eqref{eq:generator2-3terms} converges to 0. Since $f\in {\cal T}$,
the problem boils down to proving that for every $\lambda \in [0,1]$, 
$$
\lim_{n\to\infty} \  \sum_{k = 1}^{An} \frac{R(k)}{n^{1-\alpha}}\sum_{\ell = 1}^\infty \lambda^\ell |\E(\Lambda^{k,n}(\bz_n)(\ell)) - \cC^{{k}/{n}}(\bz_n)(\ell) | \ = \ 0
 $$
 and
 $$
\lim_{n\to\infty} \  \sum_{k = 1}^{An} \frac{R(k)}{n^{1-\alpha}}|\E(|\Lambda^{k,n}(\bz_n)|) - |\cC^{{k}/{n}}(\bz_n)| | \ = \ 0.
 $$
Using Lemma \ref{lemE-C},
$$
\textrm{for $\lambda <1$, } \ \sum_{k = 1}^{An} \frac{R(k)}{n^{1-\alpha}}\sum_{\ell = 1}^\infty \lambda^\ell |\E(\Lambda^{k,n}(\bz_n)(\ell)) - \cC^{\frac{k}{n}}(\bz_n)(\ell) | \ \leq \ 2  \bigg( I_{n, A} + J_{n, A}\bigg)\sum_{\ell = 1}^\infty \lambda^\ell (\ell+1)
$$
and
$$
\sum_{k = 1}^{An} \frac{R(k)}{n^{1-\alpha}}|\E(|\Lambda^{k,n}(\bz_n)|) - |\cC^{\frac{k}{n}}(\bz_n)| | \ \leq \ 2  \bigg( I_{n, A} + J_{n, A}\bigg)
$$
where 
\[  I_{n, A}\ :=\ \sum_{k=1}^{An} \frac{R(k)}{n^{2-\alpha}} \E( k p_1^{(k)})
\textrm{ and }   J_{n, A}\ :=\sum_{k=1}^{An} \frac{R(k)}{n^{1-\alpha}} \E( |\Gamma - kp_1^{(k)}|).\]
By Assumptions \ref{def:finite-xi},   and  since $\E(kp_1^{(k)}) =1$, we have
 $I_{n, A}  \sim \  \frac{1}{n} \int_0^{A}  \rho x^{-\alpha}{dx}\to 0$ as $n\to\infty$.

By point $(ii)$ of Proposition \ref{prop:moments} in the Appendix, the sequence $(k p_1^{(k)}; k\in\N)$ is uniformly integrable so that $\E( |\Gamma - kp_1^{(k)}|) \to 0$ and
by a similar integral-sum comparison, $J_{n, A} \to 0$ as $n\to\infty$.

{\bf Step 3.} 
Finally, we prove that the term on the third line of \eqref{eq:generator2-3terms} converges to 0. 
As in the previous step, it is enough to prove that for every $\lambda <1 $,
\begin{equation*}
\sum_{k = 1}^{An} \frac{R(k)}{n^{1-\alpha}}  \sum_{\ell = 1}^\infty \lambda^\ell \E\left( \left|\Lambda^{k,n}(\bz_n)(\ell)- \E(\Lambda^{k,n}(\bz_n)(\ell)) \right |\right) \to 0
\end{equation*}

and 
\begin{equation*}
\sum_{k = 1}^{An} \frac{R(k)}{n^{1-\alpha}}  \E\left( \left||\Lambda^{k,n}(\bz_n)|- \E(|\Lambda^{k,n}(\bz_n)|)) \right |\right) \to 0. 
\end{equation*}
We start by proving the first limit. Recall that $\E\left( \left|\Lambda^{k,n}(\bz_n)(\ell)- \E(\Lambda^{k,n}(\bz_n)(\ell))\right| \right) \le 2 \E(\Lambda^{k,n}(\bz_n)(\ell)) \le 2 $
since $\Lambda^{k,n}(\bz_n) \in {\cal Z}_n$, 
$\sum_{\ell = 1}^\infty \E(\Lambda^{k,n}(\bz_n)(\ell)) \leq 1$. Since $\sum_{k = 1}^{An} \frac{R(k)}{n^{1-\alpha}}\to \rho \int_0^A \frac{dx}{x^\alpha}$, it is enough to prove that for every $\ell_0$ 
\begin{equation*}
\sum_{k = 1}^{An} \frac{R(k)}{n^{1-\alpha}}  \sum_{\ell = 1}^{\ell_0} \lambda^\ell \E\left( \left|\Lambda^{k,n}(\bz_n)(\ell)- \E(\Lambda^{k,n}(\bz_n)(\ell)) \right |\right) \to 0.
\end{equation*}
By applying succesively Cauchy-Schwarz and Jensen's inequality, 
\begin{eqnarray*}
\sum_{\ell=1}^{\ell_0} \lambda^\ell \E\left( \left|\Lambda^{k,n}(\bz_n)(\ell)- \E(\Lambda^{k,n}(\bz_n)(\ell)) \right |\right) 
&\le  &  \sqrt{\sum_{\ell=1}^{\ell_0}  \lambda^\ell} \sqrt{ \sum_{\ell=1}^{\ell_0} \lambda^\ell \E\left( \left|\Lambda^{k,n}(\bz_n)(\ell)- \E(\Lambda^{k,n}(\bz_n)(\ell)) \right |\right)^2 } \\  
& \le & \sqrt{\sum_{\ell=1}^{\ell_0}  \lambda^\ell} \sqrt{ \sum_{\ell=1}^{\ell_0} \lambda^\ell {\Var(\Lambda^{k,n}(\bz_n)(\ell))}}.  
\end{eqnarray*}
Therefore, it is enough to prove that for every $\ell_0$
\begin{equation}\label{eq:finnn}
\sum_{k = 1}^{An} \frac{R(k)}{n^{1-\alpha}} \sqrt{ \sum_{\ell = 1}^{\ell_0} \lambda^\ell \Var(\Lambda^{k,n}(\bz_n)(\ell)) }\to 0.
\end{equation}
Using the first item of Lemma \ref{lemma-secondmoment1}, for every $\ell_0$ there exists a constant $C$  and a function $h(k)\to0$ such that 
\[ \forall k\leq An, \  \sum_{\ell = 1}^{\ell_0} \lambda^\ell \Var(\Lambda^{k,n}(\bz_n)(\ell))  \le C  (\frac{1}{n}+ h(k))  \]
with $h(k)\to 0$ as $k\to\infty$.
Since $\sum_{k = 1}^{An} \frac{R(k)k}{n^{2-\alpha}}\to \rho \int_0^A \frac{dx}{x^{1-\alpha}}$, it is easy to show (\ref{eq:finnn}) from there.

 The second limit can be shown along the same lines.

\end{proof}

\section{Martingale problem. Proof of Theorem \ref{Mgprob}  }
\label{sect-def-limitingprocess}
\subsection{The coagulation operator}
In this section, we study some properties of the coagulation operator ${\cal C}^x$ defined in \eqref{defCx}.

Our results will be both based on the following interpretation of the operator ${\cal C}^x$.  
Conditional on a realization of the random variable $\Gamma=w_1/\E(w_1)$ and consider the sequence of random variables
$\vec N_{x,{\bf z}}:=(N_{x,{\bf z}}(i); i\in \N)$ such that conditional on $\Gamma$, the $N_{x,{\bf z}}(i)$'s are independent  and Poisson distributed with respective parameters  $\Gamma\bz(i)/x$. 
Define $\mathbf{C}^1(\vec N_{x,{\bf z}})$ as the random vector such that $\mathbf{C}^1(\vec N_{x,{\bf z}})(\ell) = \1_{\{ \sum iN_{x,{\bf z}}(i) = \ell\}}$. Using the notations of Lemma \ref{lemma-combin},  $\mathbf{C}^1(\vec N_{x,{\bf z}})$ can be seen as the trivial coagulation operator associated to a single urn, applied to $\vec N_{x,{\bf z}}$. 
The following relation will be useful for the next results
\begin{equation}\label{CxC1}
\frac{1}{x}{\cal C}^{x}({\bf z})(\ell) = \E( \mathbf{C}^1(\vec N_{x,{\bf z}})(\ell)  ).\end{equation}

\begin{proposition}
\label{prop1}
Let $x>0$ and ${\bf z}\in{\cal Z}$. The vector ${\cal C}^x(\bz)$ is in ${\cal Z}$  and for every $\lambda\in[0,1]$,
\begin{equation}  \label{eq:cha}
{\cal C}^x {\psi}_{\lambda}(\bz ) \ = \ \E\bigg( x \exp(-|\bz| \Gamma/x)\left( \exp\left( \psi_\lambda(\bz)\Gamma /x\right) -1 \right) \bigg).
\end{equation}
\end{proposition}

\begin{proof}
We first prove that  ${\cal C}^x(\bz)$ is in ${\cal Z}$. From \eqref{CxC1}, we have
  \begin{eqnarray*}
\sum_{\ell=1}^\infty \ell {\cal C}^{x}({\bf z})(\ell)& =& x\sum_{\ell =1}^\infty \ell \P(\sum_{i=1}^\infty iN_{x,{\bf z}}(i) = \ell) \\ &=& x\E(\sum_{i=1}^\infty iN_{x,{\bf z}}(i))  = x \frac{\E(\Gamma)}{x} \sum_{i=1}^\infty i \bz(i) = 1.
\end{eqnarray*}
Define
\[\rho(\lambda) \ := \ \E \bigg( x \exp(-{|\bz|\Gamma}/{x})  (\exp\left( \psi_\lambda(\bz)\Gamma /x\right) -1) \bigg). \]
According to Lemma \ref{lem:faa} (in Appendix B), for $\ell\in\N$,
\[ \frac{1}{\ell !} \frac{d^\ell}{d \lambda^\ell} \rho(\lambda) \bigg|_{\lambda=0} \ = \   \E\left( x \exp(-{|\bz|\Gamma}/{x}) \sum_{c \in \phi^{-1}(\ell)}   \prod_{i=1}^\ell  \frac{(z(i)\Gamma/x)^{c(i)}}{c(i) !}\right) \]
whose expression coincides with the $\ell^{th}$ coordinate of ${\cal C}^x(\bz)$ in \eqref{defCx}.
In order to prove (\ref{eq:cha}), it remains to show that the Mac-Laurin expansion of $\rho$ converges to $\rho$ pointwise on in a neighborhood of $0$.
The result for $\lambda \in [0,1]$ is obtained by standard analytic continuation.
To do so, we use Taylor's theorem and prove that the remainder $R_\ell(\lambda)$ converges to 0. Let $\delta>0$. Using Lemma \ref{lem:mclaurin}, for every $\lambda<\delta$ we have
\[R_\ell(\lambda) = \int_0^\lambda \frac{\rho^{(\ell+1)}(t)}{\ell!}(\lambda -t)^\ell dt \le x\frac{(2-\delta)^{\ell}} {(1-\delta)^{2(\ell+1)}} \int_0^\lambda (\ell+1) (\lambda -t)^\ell dt\ \ge \frac{ x}{2-\delta} \left(\frac{(2-\delta)\delta} {(1-\delta)^{2}}\right)^{\ell+1}, \]
which converges to 0 as $\ell \to \infty$ for $\delta$ small enough. 

\end{proof}
The next result is useful to study the integrability of the generator $\mathcal A$.
\begin{lemma}\label{lem:endo}
For every $\bz\in{\cal Z}$  we have
$$
\vert \psi_\lambda(\bz) - {\cal C}^{x}\psi_{\lambda}(\bz) \vert \ \leq \ \frac{1}{x}  \frac{2\lambda}{1-\lambda} \ \textrm{ if  $\lambda\in[0,1)$,}
$$
and
$$
\vert \psi_1(\bz) - {\cal C}^{x}\psi_1(\bz) \vert \ \leq \ \frac{1}{x} .  
$$
\end{lemma}
\begin{proof}
 Let $x>0$ and condition on a realization of $\Gamma$. 
 \blue{ From the definition of our trivial coagulation operator, $|\mathbf{C}^1(\vec N_{x,z})| =\1_{\{|\vec N_{x,z}|>0\}}´$, 
 \begin{eqnarray}\label{eq-bornelemma8}
\mathbb{E}(|\vec N_{x,z}| -|\mathbf{C}^1(\vec N_{x,z})| ) & = & \mathbb{E}(|\vec N_{x,z}| \1_{\{|\vec N_{x,z}|\ge 2\}} \nonumber \\
 &\le & \mathbb{E}(|\vec N_{x,z}|) = \frac{\E(\Gamma)|\bz|}{x}
\end{eqnarray}
which is the desired result  for $\lambda = 1$, since $\E(\Gamma) = 1$.
}

By Lemma \ref{lemma-combin},  for $ \ell\in\N$,
\begin{eqnarray*}
 \E(| \vec{N}_{x,{\bf z}}(\ell) - \mathbf{C}^1(\vec N_{x,{\bf z}})(\ell)  |) &\leq& 2 \E(|\vec N_{x,{\bf z}}| - |\mathbf{C}^1(\vec N_{x,{\bf z}})|),
\end{eqnarray*}
which, combined with \eqref{eq-bornelemma8}, yields the result for $\lambda <1$ by summing over $\ell$.
\end{proof}

\subsection{Martingale problem}

\begin{proof}[Proof of Theorem \ref{Mgprob}]
 Lemma \ref{lem:endo}, ensures that the integral with respect to $x$ in the operator ${\cal A}$ is  integrable at $\infty$.
This, together with the fact that  $x\to x^{-\alpha}$ is integrable at $0$ shows that the operator ${\cal A}$ is well defined.
We now proceed in three steps.

{\bf Step 1.} Let $(\mu_t; t\geq0)$ be a solution to the martingale problem.
Fix $K \in \N$ and  $\vec \lambda\in \{1\}\times[0,1)^{K-1}$.
Define the projected process $(y_t^{\vec{\lambda}}; t\geq 0) := (\psi_{\vec\lambda}(\mu_t); t\geq 0)$.
\blue{In Step 1, we are going to prove the uniqueness in law of the projected process.}
Notice that, since $\lambda_1=1$, the first coordinate of $y_t^{\vec{\lambda}}$ corresponds to $|\mu_t|$.
Let ${\cal B}$ be the operator acting on $C^2{([0,1]^K)}$ such that for every $\vec y  = (y_1, \dots, y_K)  \in [0, 1]^K$
\begin{eqnarray*}  
 {\cal B} F(\vec y) 
 & := & \int_0^\infty \bigg(F\left( \E(x\exp(- y_1 \Gamma/x) (\exp( \vec{y} \Gamma/x)-1) )\right) - F(\vec{y}) \bigg) \rho x^{-\alpha}dx
 \end{eqnarray*}
 where $\exp(\vec{u})$ is the vector with coordinates $\{\exp(u_i)\}_{i=1}^K$, \blue{$\exp(\vec u ) - 1$ is the vector with coordinates $(\exp(u_i ) - 1)^K_{i=1}$} and the expected value is taken w.r.t. $\Gamma$.
 \blue{It is straightforward to see that Proposition \ref{prop1} implies that $(y_t^{\vec{\lambda}}; t\geq 0)$ satisfies this martingale problem}
 \begin{equation*} \left( F(y^{ \vec{\lambda}}_t) - \int_0^t {\mathcal B} \ F(y^ {\vec{\lambda}}_s) ds;  t\geq0\right) \ \ \ \mbox{is a martingale for every  $F\in C^2([0,1]^{K})$.} 
\end{equation*}   
 \blue{To conclude, we now show that the solution to this problem is unique.}

According to Theorem 5.1 in \cite{Bas04}, we need to check that for every $\vec \lambda\in \{1\}\times[0,1]^{K-1}$ and every measurable set  $B\subset \R \setminus\{0\}$, the function
\[\bz \to G_B(\bz) := \int_B g(\bz, x) dx : = \int_B  \frac{\left|\psi_{\vec{\lambda}}(\bz)-\cC^{x}\psi_{\vec{\lambda}}(\bz)\right|^2}{|\psi_{\vec{\lambda}}(\bz)-\cC^{x}\psi_{\vec{\lambda}}(\bz)|^2+1} \rho x^{-\alpha}dx\]
is a continuous and bounded function. 
By standard continuity theorem under the integral, this boils down to proving that for every $x \in \R\setminus \{0\}$, $\bz \to g(\bz, x)$ is continuous and that there exists a function $h$ satisfying $\int_B h(x)dx <\infty$ and  such that for every $\bz \in {\cal Z}$, $|g(\bz, x)| \leq h(x)$. 
First, observe that $ \bz  \to {\cal C}^{x}\psi_{\lambda}(\bz)$ is continuous. Since $\mathcal{C}^x (\bz)$ is defined as an expectation with respect to $\Gamma$ (see \ref{defCx}), we use again a standard continuity under the integral theorem, by noticing that the quantity inside the expectation is bounded uniformly by $x$. This implies the continuity of $\bz \to \left|\psi_{\vec{\lambda}}(\bz)-\cC^{x}\psi_{\vec{\lambda}}(\bz)\right|^2$  on $(0,\infty)$. 
The continuity of $\bz \to g(\bz, x)$ follows from there.
The existence of the upper bound $h$ follows from two observations. First, 
$$
\forall x\in (0,1]\cap B, \ \ g(\bz, x) \le  \rho x^{-\alpha}.
$$
Second, $$\forall x \in[1,\infty)\cap B, \ \ g(\bz, x) \le \sum_{i = 1}^K \left|\psi_{\lambda_i}(\bz)-\cC^{x}\psi_{\lambda_i}(\bz)\right|^2 \rho x^{-\alpha} $$
which, combined with Lemma \ref{lem:endo}, implies the existence of a constant $C$ such that for $\bz\in{\cal Z}$
$$\forall x \in[1,\infty]\cap B, \ \ g(\bz, x) \le K \frac{C^2}{x^{2+\alpha}}. $$

{\bf Step 2.} Let us study the uniqueness of the solution to our martingale problem.  
Fix $t_1<\dots<t_n$ and consider the multidimensional process  $Z:\lambda\to (\psi_\lambda(\mu_{t_1}), \dots, \psi_\lambda(\mu_{t_n}))$ on $[0,1]$ (the ``time'' parameter is now $\lambda$).  
The previous step shows that the finite dimensional distributions of $Z$ are uniquely determined.
Since $|\mu_t|<1$, the radius of convergence of $\sum_{i} \mu_t(i) \lambda^i$ is \blue{at least} $1$ and the process $Z$
is continuous a.s. This implies that the distribution of $(\psi_\lambda(\mu_{t_1}), \dots, \psi_\lambda(\mu_{t_n}); \lambda \in[0,1))$ is uniquely determined.

Finally, we can differentiate $\psi_\lambda(\mu_t)$ under the sum at $0$ infinitely many times 
to recover $\mu_t$ from its moment generating function, i.e.,
\[\forall k\in\N, \ \ \mu_t(k) = \frac{1}{k!} \frac{d^k}{d\lambda^k}\psi_\lambda(\mu_t) \vert_{\lambda=0}. \]
This shows that the finite dimensional distributions of $\mu_t$ are uniquely determined.

{\bf Step 3.} The existence of a solution follows from our convergence result (Theorem \ref{thm:conv-process}).
\end{proof}

\section{Convergence to the limiting process. Proof of Theorem \ref{thm:conv-process}}
\label{sect:convergence}
\blue{In this section, we prove convergence in  $D([0,T], {\cal Z})$ equipped with the Skorokhod $M_1$ topology. The proof is based on a useful characterization of tightness in $M_1$ (see Theorem 12.12.3 and Remark 12.3.2 in \cite{Whitt}).
We work with $M_1$ instead of the more commonly used $J_1$ because, as far as we know, it is cumbersome to apply similar arguments for the $J_1$ topology.  }
\begin{proposition}\label{prop:con-limit}
For any $T>0$, the sequence  $(\mu^n;{n \in \N})$ is tight in $D([0,T], {\cal Z})$ equipped with the Skorokhod $M_1$ topology.
\end{proposition}

\begin{proof}
\blue{Let us define the function $s: \ell^1(\R_+) \to \ell^\infty(\R_+) $ such that 
$$s(x) = \bigg(\sum_{i=1}^K x(i)\bigg)_K.$$
We know that this is a continuous function. 
We consider the process $\bar \mu^n:= s(\mu^n)$. It is sufficient to prove tightness of $\bar\mu^n$.
Observe that every entry of  $\bar \mu^n$ is decreasing.}

We use Theorem 12.12.3 in \cite{Whitt}. Let us define  the supremum norm on ${\cal Z}$ 
\blue{$$ \forall x \in {\cal Z}, \ ||x||:= \sup_{0 \le t \le T} ||x(t)|| =  \sup_{0 \le t \le T}  \max_i |x_t(i)|.   $$}
 We need to check that:
\begin{itemize}
\item[$(i)$]For each $\epsilon >0$, there exists $c$ such that 
$$\forall \ n\ge 1, \  \mathbb{P}(||\bar\mu^n|| >c ) < \epsilon.$$
\item[$(ii)$]For each $\epsilon >0$ and $\eta >0$, there exists $\delta$ such that $$\forall \ n\ge 1, \ \mathbb{P}(w(\bar \mu^n, \delta) \ge \eta)< \epsilon, $$
where $$w(x, \delta):= \sup_{t \in [0, T]}\left \{\sup_{0\vee (t-\delta) \le t_1 < t_2 < t_3 \le (t+\delta)\wedge T )} \{ ||x_{t_2} - [x_{t_1}, x_{t_3}] || \} \right \} $$
 where the segment $ [a, b]$ is defined as $ [a, b] := \{ \alpha a + (1-\alpha)b: 0 \le \alpha \le 1\}$ and the difference between $x_{t_2}$ and $[x_{t_1},x_{t_3}]$ is the smallest difference between $x_{t_2}$ and any point in the segment $[x_{t_1},x_{t_3}]$, \blue{where we use the standard definition of segments in Banach spaces}.
\end{itemize}
The first condition is trivial \blue{(since $\sum_{i=1}^K z(i) <1$ for any $z \in {\cal Z}$)}.
To check the second condition, we first notice that 
\blue{$$ ||\bar \mu^n_t|| =  | \mu^n_t |. $$
 Since every coordinate of $(\bar \mu^n_t; t\geq 0)$ is decreasing, for any $t \in [0, T]$ and any $(t-\delta) \le t_1 < t_2 < t_3 \le (t+\delta)$, 
 \begin{eqnarray*}
 ||\bar \mu^n_{t_2} - [\bar \mu^n_{t_1}, \bar \mu^n_{t_3}] ||  &\le& ||\bar \mu^n_{t_1}|| - ||\bar \mu^n_{t_3}||\\
 &\le& | \mu^n_{t_1}| - | \mu^n_{t_3}|.
\end{eqnarray*} }
Using Markov's inequality, 
\begin{eqnarray*}
\mathbb{P}((|\mu^n_{t_1}| - |\mu^n_{t_3}|) \ge \eta) &\le& \frac1{\eta}\mathbb{E}(|\mu^n_{t_1}| - |\mu^n_{t_3}|)\\
 &=& \frac1{\eta} \sum_{k=1}^\infty \frac{R(k)}{n^{1-\alpha}} \E\left( \int_{t_1}^{t_3}  (\ |\mu_u^n| - \Lambda^{k,n} |\mu_u^n|   ) du \right).
\end{eqnarray*}
Using Corollary \ref{cor:ui} and the fact that $t_3 - t_1\le \delta$, this quantity tends to 0 as $\delta \to 0$ which completes the proof.
\end{proof}

In the following, we denote by $\mu$ any subsequential limit of  $(\mu^n;{n\in \N})$  in $D([0,T], {\cal Z})$.
It remains to prove that $\mu$ is the (unique) solution to the martingale problem.
We start by showing that the limiting process $\mu$ has no fixed point of discontinuity.

\begin{lemma}\label{eq:cont-mapping}
For any $t \in [0, T]$, and $t_p\downarrow t$ \blue{or $t_p\uparrow t$},
$\mu_{ t_p} \Longrightarrow \mu_t$.
\end{lemma}
\begin{proof}
Since the function $(\nu_t:= \E(|\mu_t|); t\geq0)$ is non-increasing and valued in $[0,1]$,  it has at most a countable set $D$ of discontinuity points.  
We aim at showing that the set $D$ is empty. Let $t\in[0,T]$ and $t_p\downarrow t$. Since $D$ is countable, we can always take a sequence $\bar t_p>t_p$ s.t, $\bar t_p\notin D$ and $\bar t_p\to t$. 
By monotonicity,
$$
0 \leq \nu_t - \nu_{t_p} \leq \nu_t - \nu_{\bar t_p}.
$$
We have 
\begin{eqnarray*}
\E\bigg( |\mu^n_t| - |\mu^n_{ \bar t_p }| \bigg) & = &  \int_t^{ \bar t_p } \sum_{k=1}^\infty \frac{R(k)}{n^{1-\alpha}} \E\left (\ |\mu_u^n| - \Lambda^{k,n} |\mu_u^n|   ) \right) du.
\end{eqnarray*}
We now let $n\to\infty$ (at fixed $p$).
Since $\bar t_p \notin D$, we have $\lim_{n\to\infty} \E\bigg( |\mu^n_t| - |\mu^n_{ \bar t_p }| \bigg) \ = \nu_t - \nu_{\bar t_p}$.
On the other hand, Corollary \ref{cor:ui} implies the existence of a constant $C$ such that 
$$
\lim_{n\to\infty} \E\bigg( |\mu^n_t| - |\mu^n_{ \bar t_p }| \bigg) \leq C (\bar t_p - t).
$$
As a consequence, $\nu_t - \nu_{\bar t_p}\to0$ as $p\to\infty$.

Let us now consider the functions  $(\nu^K_t:= \E(\sum_{i = 1}^K \mu_t(i)); t\geq0)$, which are also non-increasing and valued in $[0,1]$. 
By \eqref{Nm-Cm} in Lemma \ref{lemma-combin},we have  
\[ \ \E\left(\sum_{i = 1}^K \mu_{t}^n(i) - \sum_{i = 1}^K\mu_{\bar t_p}^n(i)\right) \leq  2 K \E\left( |\mu_t^n| - |\mu_{\bar t_p}^n|\right),  \]
and we can apply the same reasoning as above to prove that for every $K\in \N$, $\nu^K_t - \nu^K_{\bar t_p}\to0$ as $p\to\infty$.

This implies  that $\mu_{t_p}$ converges to $\mu_{t}$ in distribution coordinatewise.  
By Sheffe's lemma, $\mu_{t_p}$ converges to $\mu_{t}$ in distribution in  $\ell^1(\R_+)$.

\blue{The proof for $t_p\uparrow t$ follows along the same lines. }
\end{proof}

\begin{proof}[Proof of Theorem \ref{thm:conv-process}]
We need to show that the process $(\mu_t; t\geq0)$ is the (unique) solution to the martingale problem.
Let $f\in{\cal T}$.
Let $p\in\N$. Let $h_1,\dots, h_p$ be continuous and bounded functions from $\ell^1(\R_+)$ to $\R_+$. 
Let $t_1<\dots<t_p\leq t$ and $s\geq0$.  Recall that ${\cal A}_n$ refers to the generator of the rescaled process $\mu^n$, so that  it remains to prove that for such choice of times and test functions, we have
\begin{eqnarray}\label{eq:claim-lim}\lim_{n\to\infty}& \E\left( \bigg(f(\mu^n_{t+s}) - f(\mu^n_{t})  - \int _{t}^{t+s} {\cal A}_n f(\mu^n_u) du \bigg) \prod_{i=1}^p h_i(\mu^n_{t_i})\right) \nonumber\\= & \E\left( \bigg(f(\mu_{t+s}) - f(\mu_{t})  - \int _{t}^{t+s} {\cal A} f(\mu_u) du \bigg) \prod_{i=1}^p h_i(\mu_{t_i})\right).    \end{eqnarray}

Let us now consider a coupling such that $\mu^n$ converges to $\mu$ a.s. in $D([0,T], {\cal Z})$. In virtue of Lemma \ref{eq:cont-mapping}, the times $t, t+s$ and $t_i$'s are a.s. continuity points for the limiting process so that  $\mu^n_{u}\to \mu_u$ for $u\in\{s,s+t,t_1,\cdots, t_k\}$ a.s.. 
Further, by monotonicity \blue{of each coordinate}, the set of discontinuities for the functions $(|\mu_t|; t\geq0)$ and $(\sum_{i=1}^\ell \mu_t(i)); t\geq0)$
is a (random) countable set a.s.. This implies that the set of discontinuity points for the limiting process $(\mu_t; t\geq0)$ has  a.s.  null Lebesgue measure \blue{i.e.,  for every fixed $t$, $\P(\mu \textrm{ is continuous at $t$})=1$.}

Now, in virtue of Corollary \ref{cor:ui}, we can use the bounded convergence theorem (to pass the limit inside $\E$ and the time integral). 
(\ref{eq:claim-lim}) follows from Proposition \ref{prop:cv-law} and the fact that a.s. the set of discontinuities for the limiting process has null Lebesgue measure.
\end{proof}

\section{Self-similarity }
\label{sect:selfsimilarity}
In this section, we show that the limiting process $(\mu_t; t\geq0)$ is a self-similar Markov process. 
The self-similarity property provides a natural Lamperti representation of the process, given in \eqref{eq:l-ch-t}. 
This representation allows to construct the process $\mu$ as the flow induced by a SDE driven by a L\'evy noise, see Theorem \ref{thm:flow-representation}.

Let us first generalize the concept of self-similarity in $\R_+^d$ ( \cite{Kiu, Alili}) to $\ell^1(\R_+)$.
Let $(\nu_t; t\geq0)$ be a process valued in $\ell^1(\R_+)$. Let $\P_{\bz}$ be the law of $\nu$ starting from the initial condition $\bz$.
We say that $\nu$ is self-similar with parameter $\beta$ if for every $\gamma>0$
\[\left((\nu_t; t\geq0), \P_{\bz}\right)  = \left( (\gamma \nu_{t\gamma^{-\beta} }; t\geq0 ), \P_{\gamma^{-1} \bz}\right). \]

\begin{proposition}\label{prop:self-similar}
The process $(\mu_t; t\geq0)$ is a self-similar process with parameter $\beta=\alpha-1$.
\end{proposition}
\begin{proof}
Fix $\gamma>0$ and consider the rescaled process $\mu^{(\gamma)}:= (\gamma \mu_{t\gamma^{-\beta} }; t\geq0 )$. By uniqueness of the solution of  the martingale problem
introduced in Theorem \ref{Mgprob}, it is sufficient to check that $\mu^{(\gamma)}$ is also a solution.
Fix an integer $K$, a function $F\in C^2([0,1]^K)$ and a vector $\vec{\lambda}\in[0,1]^K$. 
Let $F^{(\gamma)}(x) = F(\gamma x)$. Since $(\mu_s; s\geq0)$ is \blue{a} solution of the martingale problem, then
\begin{eqnarray*} M_t  & := &  F^{(\gamma)}( \psi_{\vec\lambda}(\mu_t) ) \ - \int_0^t\int_{\R_+} \bigg(F^{(\gamma)}( \psi_{\vec \lambda}( \cC^{x}(\mu_s) )) -  F^{(\gamma)}( \psi_{\vec\lambda}(\mu_s) )  \bigg)  \rho x^{-\alpha}dx ds\\
& = &  F( \psi_{\vec\lambda}( \gamma \mu_t) ) \ - \int_0^t\int_{\R_+} \bigg(F( \gamma \psi_{\vec \lambda}( \cC^{x}(\mu_s) )) -  F( \psi_{\vec\lambda}(\gamma \mu_s) )  \bigg) \rho x^{-\alpha}dxds
\end{eqnarray*}
is a martingale. Next, we observe that for every $\bz\in\ell^1(\R_+)$
\begin{eqnarray*}
\gamma \cC^{x}( \psi_{\vec\lambda}(\bz)) & = & \cC^{\gamma x} ( \psi_{\vec\lambda}(\gamma\bz)). 
\end{eqnarray*}
This implies  
\[ M_{t \gamma^{-\beta}} \ = \  F( \psi_{\vec\lambda}( \mu_t^{(\gamma)}) ) \ - \int_0^{t \gamma^{-\beta}}\int_{\R_+} \bigg(F(  \psi_{\vec \lambda}( \cC^{\gamma x}(\gamma \mu_s) )) -  F( \psi_{\vec\lambda}(\gamma \mu_s) )  \bigg)  \rho x^{-\alpha}dxds.  \]
Changing the variables $\bar s \gamma^{-\beta} = s$ and $\bar x= \gamma x$ in the latter integral yields 
\[ M_{t \gamma^{-\beta}} \ = \ F( \psi_{\vec\lambda}( \mu_t^{(\gamma)}) ) \ - \underbrace{\gamma^{-\beta + \alpha-1}}_{=1} \int_0^{t}\int_{\R_+} \bigg(F(  \psi_{\vec \lambda}( \cC^{\bar x}( \mu^{(\gamma)}_{\bar s}) )) -  F( \psi_{\vec\lambda}( \mu^{(\gamma)}_{\bar s}) )  \bigg)  \rho\bar x^{-\alpha}d\bar xd\bar s. \] 
This shows that
\[  F( \psi_{\vec\lambda}( \mu_t^{(\gamma)}) ) \ - \int_0^{t}\int_{\R_+} \bigg(F(  \psi_{\vec \lambda}( \cC^{\bar x}( \mu^{(\gamma)}_{\bar s}) )) -  F( \psi_{\vec\lambda}( \mu^{(\gamma)}_{\bar s}) )  \bigg)  \rho\bar x^{-\alpha}d\bar xd\bar s \] 
defines a martingale for every $F\in C^2([0,1]^K)$ so that $\mu^{(\gamma)}$ is also \blue{a} solution of the martingale problem introduced in Theorem \ref{Mgprob}.
\end{proof}

\begin{proof}[Proof of Theorem \ref{thm:lamperti}]
By Proposition \ref{prop:self-similar} the process is self-similar in $\ell^1(\R_+)$. Theorem 2.3 in \cite{Alili} \blue{is analogous to} Theorem \ref{thm:lamperti} for self-similar processes valued in $\R^d$.
The proof goes verbatim for our state space. 
\end{proof}

Now, let us provide an alternative representation of the MAP.
Consider the generating function of $\theta_t$ (the component encoding the asymptotic frequencies \blue{introduced in Theorem \ref{thm:lamperti}}), i.e.,
\begin{equation} \label{defxlambda}
\forall t\geq0, \ \ x_t^{\lambda} := \psi_\lambda(\theta_{t})= \sum_{i\geq1} \lambda^i \theta_t(i).
\end{equation}
For every $\lambda \in [0, 1), \vec\lambda \in [0,1)^K$, define the processes $x^\lambda:=(x^\lambda_t;t\geq0)$ and $(x_t^{\vec{\lambda}}; t\geq 0) := (\psi_{\vec\lambda}(\theta_t); t\geq 0)$.
By construction, we have $x_0^\lambda= \lambda$.
We consider the flow of stochastic processes indexed by $\lambda$, $(x^\lambda; \lambda \in [0,1))$.
 In Theorem \ref{thm:flow-representation} we prove that $((\xi_t, x_t^{\lambda}); t\geq 0)$ has the same law as the flow of stochastic processes $((\bar \xi_t,\bar x_t^{\lambda}); t\geq 0)$ defined in \eqref{eq:flow} below.

Let $N$ be a Ppp in \blue{$\mathbb R_+\times \mathbb R_+$} with intensity measure $dt\otimes \rho a^{-\alpha}da$.  
We define 
\[H(x,a) \ := \ \frac{ \E\bigg(\exp(-\Gamma / a)(\exp( \Gamma x/a)-1)\bigg)}{\E\bigg(1-\exp(-\Gamma / a)\bigg)}, \ \ g(a) \ = \  -\log\E \left[a \bigg(1-\exp(-\Gamma  / a)\bigg)\right], \] 
\blue{where the expected value is taken w.r.t. $\Gamma$.}
For every $\lambda \in[0,1)$, we consider \blue{the unique weak solution to the following stochastic flow}
\begin{eqnarray}\label{eq:flow}
\left \{ \begin{array}{lll}
\bar x_t^\lambda \ &= \ \lambda  \ + &  \ \int_{\R_+^*}\int_0^t  \left[H (\bar x_{s-}^\lambda, a) - \bar x_{s-}^\lambda\right]N(ds \times da)\\
\bar \xi_t \ &= \    & \ \int_{\R^*_+}\int_0^t  g(a) N(ds \times da).
\end{array} \right.
\end{eqnarray}
For every $\vec \lambda \in [0, 1)^K$, define the projected process $(\bar x_t^{\vec{\lambda}}; t\geq 0) := ((\bar x_t^{{\lambda}_1}, \dots, \bar x_t^{{\lambda}_K}); t\geq 0)$.
Note that since $\E(\Gamma) = 1$, $(\bar \xi_t; t\geq0)$ is a spectrally positive subordinator. 
\blue{The existence and uniqueness of such a weak solution is proved along the exact same lines as in \cite{BertoinLeGall2} Theorem 2, where they considered qualitatively similar stochastic flows driven by a Levy noise.}

\begin{theorem}\label{thm:flow-representation}
Let $K\geq2$. For every $\vec{\lambda}' \in [0,1)^{K-1}$,  let $\left(( \xi_t, x^{\vec{\lambda'}}_t); t\geq0 \right)$  be defined as in Theorem \ref{thm:lamperti} and \eqref{defxlambda} and  let $\left(( \bar \xi_t, \bar x_t^{\vec{\lambda}'}); t\geq0 \right)$ be \blue{the unique weak solution to} the stochastic flow defined in \eqref{eq:flow}.
Then
\[\left((\xi_t, x^{\vec{\lambda}'}_t); t\geq0 \right)  \ = \ \left(( \bar \xi_t, \bar x^{\vec{\lambda}'}_t); t\geq0 \right) \ \ \mbox{in law}.\]
\end{theorem}

\blue{This result allows us to identify the jump measure of the subordinator $\xi_t$ as follows. Consider the measure on $\R_+$ defined by $\rho x^{-\alpha} dx$. The jump measure of the subordinator is the pushforward of this measure by the function $g$ (as defined above). This is a direct consequence of Theorem \ref{thm:flow-representation}.}

\begin{proof}
Let $\vec{\lambda} = (1, \vec{\lambda}') \in \{1\}\times[0,1)^{K-1}$.
Let $(\mu_t; t\geq0)$ be a solution of the martingale problem defined in \eqref{MP}. Recall the definition of 
$$y_t^{\vec{\lambda}} := \psi_{\vec\lambda}(\mu_t) = (|\mu_t|, \psi_{\vec{\lambda}'}(\mu_t))$$ 
(as in the proof of Theorem \ref{Mgprob}). For every $\vec y\in[0,1]^K$, recall the definition of 
\begin{eqnarray*}  
 {\cal B} F(\vec y) 
 & := & \int_0^\infty \bigg(F\left( \E(x\exp(- \vec{y}_1 \Gamma/x) (\exp( \vec{y} \Gamma/x)-1 ))\right) - F(\vec{y}) \bigg) \rho x^{-\alpha}dx,
 \end{eqnarray*}
 where the expected value is taken w.r.t. $\Gamma$. As in the proof of Theorem \ref{Mgprob}, $y^{\vec{\lambda}}$
 is the unique solution of the martingale problem
\begin{equation*} \left( F(y^{ \vec{\lambda}}_t) - \int_0^t {\mathcal B} \ F(y^ {\vec{\lambda}}_s) ds;  t\geq0\right) \ \ \ \mbox{is a martingale for every  $F\in C^2([0,1]^{K})$.} 
\end{equation*}

\medskip

Define the process $(\bar w^{\vec \lambda}_t; t\geq0):=(\exp(-\bar \xi_{t}), \exp(-\bar \xi_{t})\bar x_{t}^{\vec{\lambda'}}; t\geq0 )$. 
Let $(\bar \tau_t; t \geq0)$ be the Lamperti change of time in (\ref{eq:l-ch-t}) defined w.r.t. $\bar \xi$. Since $\bar \tau$ is the inverse time change defined  in Theorem \ref{thm:lamperti},
we need to prove that $(\bar w^{\vec \lambda}_{\bar \tau_t}; t\geq0) = (y^{\vec \lambda}_t; t\geq0)$ in law. The strategy consists  in showing that the time changed process $w_{\bar \tau_t}^{\vec{\lambda}}$
solves the same martingale problem.

For every $\vec w = (w_1, \dots, w_K) \in [0,1]^K$, define
\begin{eqnarray*}  
 {\cal D} F(\vec w) 
 & := & \int_0^\infty \bigg(F(w_1 \exp(-g(a))\bar H(\vec{w}, a)) - F(\vec{w})\bigg) \rho a^{-\alpha}da,
 \end{eqnarray*}
where \blue{$\bar H(\vec {w}, a) := (1, H(w_2/w_1, a),\dots, H(w_K/w_1, a)) $}.
From the definition of the stochastic flow \eqref{eq:flow}, at every jump time $(t, a)$, we have the following transitions
\[\bar \xi_{t^+}  = \bar \xi_{t^-} + g(a), \ \ \bar x^{\lambda}_{t^+} =  H(\bar x_{t^-}^{\lambda}, a).\] 
\blue{By applying Ito's formula in the discontinuous case (see \cite{Meyer})}
\begin{equation*} \left( F(\bar w^{ \vec{\lambda}}_t) - \int_0^t {\mathcal D} \ F(\bar w^ {\vec{\lambda}}_s) ds;  t\geq0\right) \ \ \ \mbox{is a martingale for every  $F\in C^2([0,1]^{2})$.} 
\end{equation*}  

 Since ${\bar \tau}_t$ is a stopping time,  the time changed process
\begin{equation}\label{eq:11}
\bigg (F(\bar w_{{\bar \tau}_t}^{\vec\lambda})  \ -  \ \int_0^{{\bar \tau}_t} {\mathcal D} \ F(\bar w^ {\vec{\lambda}}_s) ds , t\geq 0\bigg) 
\end{equation}
is also a martingale \blue{with respect to the time-changed filtration $\bar{\mathcal{G}}_{\tau_t}$, where $\bar{\mathcal{G}}_t$ is the filtration of the original weak solution}.
We now make the change of variable $s=\ \bar \tau_u= \inf\{t>0  :  \int_0^t \exp(\beta \xi_v) dv>u\}$ in the latter integral. Since
$u=\int_0^{s} \exp((\alpha-1) \bar \xi_v) dv$, we have
\begin{equation}\label{lampertitc}
\exp(- (\alpha-1)  \bar \xi_{\bar \tau_u}) du =  ds. \end{equation}
 Thus, we get that
\begin{equation*}
\int_0^{{\bar \tau}_t} {\mathcal D}  F(\bar w^ {\vec{\lambda}}_s) ds =
\int_0^{t}\int_0^\infty\exp(- (\alpha-1)\bar \xi_{\bar \tau_u} )   \bigg(F\left( e^{- \bar \xi_{\bar \tau_u}  -g(a)} \bar H( \bar w^{\vec \lambda}_{\bar \tau_u} ,a) \right)  -  F( \bar w_{\bar \tau_u}^{\vec\lambda}  ) \bigg) \rho a^{-\alpha}dad u .
 \end{equation*} 
Making the change of variables $\bar a= a/ e^{-\bar \xi_{\tau_s}}, s =u$, we get that
\begin{eqnarray*} 
\int_0^{{\bar \tau}_t} {\mathcal D} F(\bar w^ {\vec{\lambda}}_s) ds
&=& \int_0^{t}\int_0^\infty \bigg(F\left(e^{- \bar\xi_{\bar \tau_u} -g(\bar a e^{\bar \xi_{\bar\tau_u}} )}\bar H(\bar w^\lambda_{\bar\tau_u}, \bar a e^{\bar \xi_{\bar\tau_u}})  \right) - F\left(\bar w^{\vec\lambda}_{\bar\tau_u}\right) \bigg) \rho \bar a^{-\alpha} d\bar a du \\
&=&   \int_0^{t}\int_0^\infty \bigg(F\left( \E\left(a\exp(-\frac{e^{-\bar \xi_{\bar \tau_u}}\Gamma}{a}) ( \exp(\frac{\bar w^{\vec\lambda}_{\bar\tau_u}\Gamma}{a})-1 )\right) \right) - F\left(\bar w^{\vec\lambda}_{\bar\tau_u}\right) \bigg) \rho a^{-\alpha} da du \\
&=& \int_0^{t}  {\mathcal B} \ F(\bar w^ {\vec{\lambda}}_{\bar\tau_u}) du,
\end{eqnarray*} 
\blue{where the second equality is obtained by using the definitions of $\bar H$ and $g$.}
This completes the proof of the proposition since (\ref{eq:11}) is a martingale.
\end{proof}

\section{Site frequency spectrum : Proof of Theorem \ref{thm:cv-sfs}  }
\label{sect-SFS}
In this section, we use the Poissonian construction of a general Dirichlet coalescent. 
Namely, we consider a Ppp on $\R_+\times \N$ with intensity measure $dt \otimes \sum_{k\in\N} R(k)\delta_k$. 
For $k\geq1$, an atom $(t,k)$ of the point process corresponds to an event with $k$ urns in the paintbox construction. It will be referred to as a $k$-event. We recall the definition of the non-rescaled process $(\hat \mu_t^n;t\geq0)$, where $\hat \mu_t^n(i)$ counts the number of blocks of size $i$ at time $t$ and  $ \mu_t^n = \frac1n \hat \mu_{tn^{\alpha-1}}^n$ in \eqref{muresc}.
As a consequence of Theorem \ref{thm:conv-process}, for any $T>0$,
\[ (\mu_t^n; t\geq0) \ \Longrightarrow  \ ( \mu_t; t\geq0) \ \textrm{in $D([0,T], {\cal Z} ).$}\] 
(In particular, $(|\mu_t^n|;t\ge0)$ converges to the total mass of the limiting process $(|\mu_t|;t\ge0)$.) By the Skorokhod representation theorem (Theorem 6.7 in \cite{Billingsley}), we assume without loss of generality that  the convergence holds almost surely. Under this coupling, we aim at proving the following result.
\begin{proposition}\label{thm:11}
Let $\hat {\bf T}_n=\inf\{t >0 \ : \  |\hat \mu^n_t|=1\}$  be the time to the {MRCA} of the population and let $ {\bf T}_n=n^{1-\alpha} \hat{\bf T}_n$ be its renormalized version.
As $n\to\infty$, 
\begin{equation*}
\int_0^{ {\bf T}_n } \mu_s^n  ds \ \Longrightarrow \ \int_0^{\infty} \mu_s   ds \ \ \ \mbox{in $\ell^1(\R^+)$.} 
\end{equation*}
\end{proposition}

The latter  mostly relies on the following technical lemma.

\begin{lemma}\label{lem:bd}
We have that
\[\limsup_{n} \E(\int_0^{{\bf T}_n}|\mu^n_s| ds) <\infty. \]
\end{lemma}
\begin{proof}
Set $\tau^n_0=0$ and define inductively 
\[\forall j\ge1, \  \tau^n_j = \inf\{t>\tau_{j-1}^n \ : \ \mbox{ a $k$-event occurs with $k\leq \frac{n}{2^j} $ } \}.\]
To simplify the notation, we drop the dependence in $n$ and write  $\tau_j^n\equiv \tau_j$.
By the strong Markov property, $\tau_j-\tau_{j-1}$ is an exponential random variable with parameter
$\sum_{k\leq\frac{n}{2^j}} R(k)$. 
Recall that  $\lim_{k\to\infty} k^{\alpha} R(k) = \rho$.
The limiting behavior of $R$ implies that
\[ \lim_{\ell\to\infty}  \frac{1}{\ell^{1-\alpha}}\sum_{k\leq \ell} R(k)  \ = \int_0^1 \frac{\rho dx}{x^\alpha} = \frac{\rho}{1-\alpha}.  \]
 Let $\ell_0=\inf\{k\geq0: \ R(k)> 0\}$ .
Since $\sum_{k\leq \ell} R(k)>0$ for every $\ell\geq \ell_0$, 
there exists $C>0$ such that for every $\ell\geq \ell_0$ 
\[ \sum_{k\leq\ell} R(k)  \ \geq  \ C^{-1} \ell^{1-\alpha}.  \]
Define
$j_0^n \equiv j_0 = [\log_2(\frac{n}{\ell_0})]$. For every $j\leq j_0$, we have $n/2^j \geq \ell_0$ so that 
\[\sum_{k\leq\frac{n}{2^j}} R(k) \geq C^{-1}\left(\frac{n}{2^j} \right)^{1-\alpha} \ \ 
\mbox{and thus $\E\left(\tau_j-\tau_{j-1}\right) \leq C \left(\frac{2^j}{n} \right)^{1-\alpha}$}. \]
By construction, $|\hat \mu^n_t|\leq \frac{n}{2^j}$ for every $t>\tau_j$.
As a consequence, for every $j \leq j_0$, we have
\begin{eqnarray} \label{eq:bound-int-mut}
\E\left(\int_0^{\tau_j} |\hat \mu^n_t| \right) & \leq  & \sum_{i=1}^{j} \E(\tau_{i} -\tau_{i-1}) \frac{n}{2^{i-1}}\nonumber \\
& \leq  & C \sum_{i=1}^j \left(\frac{n}{2^{i-1}}\right)^\alpha\nonumber \\
& \leq & 2 C n^{\alpha}.
\end{eqnarray}

Next, we use the following change of variables
\begin{eqnarray}
\frac{1}{n^{\alpha}} \int_0^{\hat {\bf T}_n} |\hat \mu^n_s| ds  =  \frac{1}{n^{\alpha}} \int_0^{n^{1-\alpha} \hat {\bf T}_n } | \hat \mu^n_{\bar s/n^{1-\alpha}}| \frac1{n^{1-\alpha}}d\bar s   =   \int_0^{{\bf T}_n} |\mu^n_{\bar s}| d\bar s.
\label{change-var}
\end{eqnarray}

Using the fact that there are at most $\ell_0$ lineages remaining at time $\tau_{j_0}$, we have
\begin{eqnarray*}
\E\left(\int_0^{{\bf T}_n} |\mu_t^n| dt \right) & =  & \frac{1}{n^{\alpha}} \E\left(\int_0^{\hat{\bf T}_n}  |\hat \mu^n_t| dt\right) \nonumber\\
& \leq & \frac{1}{n^{\alpha}} \E\left(\int_0^{\tau_{j_0}}  |\hat \mu^n_t| dt\right) \ + \ \frac{1}{n^{\alpha}}\E\left(\int_{\tau_{j_0}}^{\hat {\bf T}_n} |\hat \mu^n_t| dt\right) \nonumber\\
& \leq & 2C \ +\ \frac{1}{n^{\alpha}}  \E\left( \int_0^{\hat {\bf T}_{\ell_0}} |\hat \mu^{\ell_0}_t| dt\right )\nonumber \\
& \leq & 2C \ +\ \frac{\ell_0^\alpha}{n^{\alpha}}  \E\left( \hat {\bf T}_{\ell_0} \right),
\end{eqnarray*}
where in the last line we used \eqref{eq:bound-int-mut}.
It remains to show that the expectation on the RHS is finite. In order to see that, we note that successive $\ell_0$-events 
are separated by independent exponential r.v.'s with the same parameter $R(\ell_0)>0$. Further, at any of those events, there is a strictly positive probability $p$ to go from $n$ lineages to a single lineage. By a simple coupling argument, one can bound from above the r.v.
$\hat {\bf T}_{\ell_0}$ by $\sum_{i=1}^X e_i$ where the $e_i$'s are i.d.d. exponential r.v.'s with parameter $R(\ell_0)$
and  $X$ is an independent geometric r.v. with parameter $p$. Since the upper bound has mean $\frac{1}{R(\ell_0) p}<\infty$, $\E(\hat {\bf T}_{\ell_0})<\infty$.
\end{proof}

\begin{proof}[Proof of Proposition \ref{thm:11}]
We need to show that for every $A>0$,
\begin{equation}\label{eq:A1}
\limsup_{A\to\infty} \limsup_{n\to\infty} \E( \int_A^{{\bf T}_n}|\mu^n_s| ds)  =  0 ,
\end{equation}
\blue{where, implicitly, $\int_A^{{\bf T}_n} |\mu^n_s|ds =0$ on the event $\{{\bf T}_n <A\}$ (which is highly unlikely in the limit)};
for every $d\in\N$,  \begin{equation}\label{eq:Abis}
 \int_0^{A}(|\mu^n_s|, \mu^n_s(1), \dots, \mu^n_s(d))  ds \Longrightarrow  \int_0^{A} (|\mu_s|, \mu_s(1), \dots, \mu_s(d))   ds   ,
\end{equation}
and   \begin{equation}\label{eq:A2}
 \E(\int_0^\infty|\mu_s| ds)  < \infty.
\end{equation}
We start with (\ref{eq:Abis}). Let $K\in\N$. By monotonicity, we have
$$
\int_0^A |\mu_s^n| ds \leq \frac{1}{K} \sum_{i=0}^{\lfloor AK\rfloor+1} |\mu^{n}_{\frac{i}{K}}| .
$$
Since the limiting process $(\mu_t; t\geq0)$ has no fixed point of discontinuity (see Lemma \ref{eq:cont-mapping}), the RHS is converging to $ K^{-1}\sum_{i=0}^{\lfloor AK\rfloor+1} |\mu_{{i}/{K}}|$ for every $K\in\N$.
As a consequence, $\lim \int_0^A |\mu_s^n| ds$ is bounded from above by $\int_0^A |\mu_s| ds$. A similar argument shows the reverse bound. This proves that 
$\lim_n \int_0^A |\mu_s^n| ds = \int_0^A |\mu_s| ds$. 
Next, for every $J,n\in\N$, the process $(\sum_{i=1}^J \mu^n_s(i); s\geq0)$ is also monotone in $s$, the exact same argument shows 
convergence of $\sum_{i=1}^J \int_0^A \mu^n_s(i) ds$ to   $\sum_{i=1}^J \int_0^A \mu_s(i) ds$. Finally, the proof of (\ref{eq:Abis}) is complete by noting that all the previous convergence statements hold jointly for every $J\in\N$.

Let us now proceed with the first property (\ref{eq:A1}). 
The Markov property, together with the change of variables \eqref{change-var} implies that
\begin{eqnarray*}
\E( \int_A^{{\bf T}_n}|\mu^n_s| ds) 
& = & \frac{1}{n^\alpha} \sum_{j=1}^n \P\bigg(| \mu^n_A| = \frac{j}{n} \bigg) \E\left(\int_0^{\hat{\bf T}_j}  |\hat\mu^j_s| ds \right)   \\
& \leq &    \frac{C}{n^\alpha} \sum_{j=1}^n  j^\alpha \P\bigg(| \mu^n_A| = \frac{j}{n} \bigg)    
\end{eqnarray*}
where $C$ is a constant and the inequality follows from \eqref{eq:bound-int-mut}. Let $\varepsilon\in(0,1)$. Then 
\begin{eqnarray*}
\E( \int_A^{{\bf T}_n}|\mu^n_s| ds) & \leq  &    \frac{C}{n^\alpha} 
\bigg( \sum_{j=\lfloor\varepsilon n\rfloor+1}^n  j^\alpha \P\bigg(| \mu^n_A| = \frac{j}{n} \bigg) \ 
+ \  \sum_{j=1}^{\lfloor\varepsilon n\rfloor}  j^\alpha \P\bigg(| \mu^n_A| = \frac{j}{n}  \bigg)   \bigg) \\
& \leq & C\bigg(  \P\bigg( | \mu^n_A|\in [\varepsilon,1] \bigg)  \ 
+ \  \varepsilon^\alpha \bigg) . 
\end{eqnarray*}
For every arbitrary $\varepsilon\in(0,1)$,
\begin{eqnarray*}
\limsup_{A\to\infty} \limsup_{n\to\infty} \E( \int_A^{ {\bf T}_n}|\mu^n_s| ds) 
& \leq &  C\bigg(  \lim_{A\to\infty} \P\bigg( |\mu_A| \in [\varepsilon,1] \bigg)  \ 
+ \  \varepsilon^\alpha \bigg) \\
& = & C \varepsilon^\alpha,
\end{eqnarray*}
where we used the fact that $|\mu_t|\to0 $ as $t\to\infty$ (this can easily be proved from the Lamperti transform).
This ends the proof of (\ref{eq:A1}).

We now turn to (\ref{eq:A2}). First,
recall that ${\bf T}_n>{\bf T}_2>0$ then $\hat {\bf T}_n\to \infty$. Furthermore, the fact that $|\mu_t^n| \to |\mu_t|$ and Fatou's lemma yield
\begin{eqnarray*}
\E(\int_0^\infty|\mu_s| ds ) & \leq & \lim_{n\to\infty} \E( \int_0^{\hat {\bf T}_n}|\mu^n_s| ds).
\end{eqnarray*}
So, (\ref{eq:A2}) follows from Lemma \ref{lem:bd}.
\end{proof}

\begin{proof}[Proof of Theorem \ref{thm:cv-sfs}]

By Proposition \ref{thm:11},
\begin{equation*}
\int_0^{ {\bf T}_n }\mu^n_s ds \ \Longrightarrow \ \int_0^{\infty} \mu_s ds \ \ \mbox{ \ in $\ell^1(\R_+)$.\ }
\end{equation*}
The relation
\[ \int_0^\infty \mu_s ds \ = \int_0^\infty \theta_{\tau_s} \exp(-\xi_{\tau_s}) ds\ = \  \int_0^\infty \theta_u \exp((\alpha-2) \xi_u) du\]
follows from the change of variable $\tau_s=u$ and \eqref{lampertitc}.

 Given a realization of the coalescent, the  number of segregating mutations in a branch sub-tending $i$ leaves is given by a Poisson r.v. with parameter $r \int_0^{\hat {\bf T}_n} \hat \mu^n_s(i) ds$. The total number of mutations is given by a Poisson r.v. with parameter $r \int_0^{\hat {\bf T}_n} |\hat \mu^n_s| ds$.
Using the change of variables  \eqref{change-var},
\begin{eqnarray*}
\frac{r}{n^{\alpha}} \int_0^{\hat {\bf T}_n} \hat \mu^n_s ds 
& = & r \int_0^{n^{1-\alpha} \hat {\bf T}_n } \mu^n_{\bar s} d\bar s \ \Longrightarrow r \int_0^\infty \mu_{\bar s} d\bar s,
\end{eqnarray*}
which implies the desired result.
\end{proof}

\section*{Appendix A: Moment estimates of the partition mass components.}
\label{app:moments}

In this section we provide some useful asymptotic results for the lengths of the unit mass partition when they are obtained from i.i.d. random variables with a finite second moment. 
\begin{proposition}\label{prop:moments}
Consider a family of i.i.d. random variables $(w_1,\dots,w_k)$ on $(0, \infty)$ such that $\E(w_1^2)<\infty$.
For $j\in[k]$, set $p^{(k)}_j=w_j/\sum_{i=1}^kw_i$ 
and $
\bar p_j^{(k)} = {w_j}/{k \E(w_1)}.
$
Then we have the following asymptotics:
\begin{itemize}
\item[(i)] $\lim_{k \to\infty}k p^{(k)}_1 = \Gamma := \frac{w_1}{\E(w_1)} $ almost surely.
\item[(ii)] $\lim_{k \to\infty} \E((kp^{(k)}_1)^2)= \frac{\E(w_1^2)}{\E(w_1)^2}.$
\item[(iii)] 
$ \lim_{k \to \infty} \E\left( (k \bar p_1^{(k)}  - kp_1^{(k)})^2 \right)  = 0. $
\end{itemize}
 \end{proposition}
\begin{proof}
 By the law of large numbers we obtain $(i)$.
The result in $(ii)$  can be found in Lemma 4.3 of \cite{cortines}.

To prove $(iii)$, we write
\begin{eqnarray*}
 \E\left( (k \bar p_1^{(k)}  - kp_1^{(k)})^2 \right) &= & \E\left( \frac{k^2w_1^2}{(\sum_{i = 1}^k w_i)^2} \right) -  \E\left( \frac{2kw_1^2}{\E(w_1)\sum_{i = 1}^k w_i} \right) +  \E\left( \frac{w_1^2}{\E(w_1)^2} \right).
\end{eqnarray*}
Using $(ii)$, it is enough to prove that 
\begin{eqnarray*}
 \E\left( \frac{kw_1^2}{\E(w_1)\sum_{i = 1}^k w_i} \right) \to  \frac{\E(w_1^2)}{\E(w_1)^2}.
\end{eqnarray*}
We follow the proof of  Lemma 4.3 of \cite{cortines}.
Let $\delta_1>0$, then
\[ \E\left(\frac{k w_1^2 }{\E(w_1)\sum_{i=1}^k w_i }\right)  \ \ge \  \E\left(\frac{ w_1^2 }{\delta_1 + \E(w_1)(k^{-1}\sum_{i=1}^k w_i )}\right). \]
Since $w_1>0$ almost surely, the bounded convergence theorem yields
\[\liminf_{k \to\infty}\E\left(\frac{k w_1^2 }{\E(w_1)\sum_{i=1}^k w_i )}\right)  \ \ge \ \frac{\E(w_1^2)}{\delta_1 + \E(w_1)^2}.  \]
The inequality holds for every $\delta_1>0$, which implies that the above limit is larger than $\E(w_1^2)/ \E(w_1)^2$.

To prove an upper bound, we use the Cauchy-Schwarz inequality,
\begin{eqnarray*}
 \E\left( \frac{kw_1^2}{\E(w_1)\sum_{i = 1}^k w_i} \right) \le \frac{1}{\E(w_1)} \sqrt{\E(w_1^2)}\sqrt{\E\left(\frac{w_1^2}{(k^{-1}\sum_{i = 1}^k w_i)^2}\right)}
\end{eqnarray*}
and the result follows from $(ii)$.
\end{proof}

\begin{lemma}
\label{lemma-lipschitz}
Let $g_1, g_2$ be two Lipschitz functions with respective Lipshitz constants $G_1, G_2\ge1$.  Also suppose that $|| g_1||_\infty, || g_2||_\infty\le1$.
Define $p_1^{(k)}, p_2^{(k)}$ as in Proposition \ref{prop:moments}. Then, 
\[
|\Cov(g_1(k p_1^{(k)}) , g_2(k p_2^{(k)}) )| \le \max\{G_1, G_2,  G_1 G_2\} h(k),
\]
where $h$ is a function of $k$, independent on $g_1, g_2$ going to 0 as $k\to \infty$.
\end{lemma}

\begin{proof}
Set
$
\bar p_1^{(k)} := \frac{w_1}{k \E(w_1)}
$ and $
\bar p_2^{(k)} := \frac{w_2}{k \E(w_2)}
$.
We have
\begin{eqnarray*}
&& | \Cov (g_1 (k p_1^{(k)}), g_2(k p_2^{(k)}))| \\
&& = | \Cov \bigg( g_1 (k \bar  p_1^{(k)}) + ( g_1 (k  p_1^{(k)}) -g_1 (k \bar  p_1^{(k)}) ),  g_2 (k \bar  p_2^{(k)}) + ( g_2 (k \bar  p_2^{(k)})  -g_2 (k \bar  p_2^{(k)})  )  \bigg) |\\
&& \le   \underbrace{|\Cov \bigg(  g_1 (k  p_1^{(k)}) -g_1 (k \bar  p_1^{(k)}) ,  g_2 (k  p_2^{(k)}) -g_2 (k \bar  p_2^{(k)}) \bigg)|}_{A}  \\
&& + \underbrace{|\Cov\bigg(g_1 (k \bar  p_1^{(k)}),  g_2 (k  p_2^{(k)}) - g_2 (k \bar  p_2^{(k)})   \bigg)|}_{B}  \ + \ \underbrace{|\Cov\bigg(g_2 (k \bar  p_2^{(k)}),  g_1 (k  p_1^{(k)}) -g_1 (k \bar  p_1^{(k)})   \bigg)|}_{C}.
\end{eqnarray*}
To bound $A$, observe that $p_1^{(k)}$ (resp. $\bar p_1^{(k)}$) has the same law as $p_2^{(k)}$ (resp. $\bar p_2^{(k)}$). So, by the Cauchy-Schwarz inequality,
\begin{eqnarray*}
 A & \leq &   \sqrt{\Var( g_1 (k  p_1^{(k)}) -g_1 (k \bar  p_1^{(k)}))} \sqrt{\Var( g_2 (k  p_1^{(k)}) -g_2 (k \bar  p_1^{(k)})  )}\\
&  \leq & \sqrt{ \E\bigg(( g_1 (k  p_1^{(k)}) -g_1 (k \bar  p_1^{(k)})  )^2\bigg)} \sqrt{\E \bigg((g_2 (k  p_1^{(k)}) -g_2 (k \bar  p_1^{(k)}) )^2 \bigg)}  \\ 
& \leq & G_1 G_2 \E\bigg( ( k \bar p_1^{(k)}  - k  p_1^{(k)} )^2 \bigg)  = G_1 G_2 h(k), 
\end{eqnarray*}
where in the last line we used point $(iii)$ of Proposition \ref{prop:moments}.

Let us turn to $B$ (the bound of  $C$ will be obtained in the same way). 

\begin{eqnarray*}
 B & \leq &   \sqrt{\Var( g_1 (k \bar  p_1^{(k)}))} \sqrt{\Var( g_2 (k  p_2^{(k)}) -g_2 (k \bar  p_2^{(k)})  )}\\
&  \leq & \sqrt{ \E\bigg(g_1 (k \bar  p_1^{(k)}  )^2\bigg)} \sqrt{\E \bigg((g_2 (k  p_2^{(k)}) -g_2 (k \bar  p_2^{(k)}) )^2 \bigg)}  \\ 
& \leq & || g_1||_\infty G_2 \sqrt{E\bigg( ( k \bar p_2^{(k)}  - k  p_2^{(k)} )^2} \bigg)  =  G_2  h(k), 
\end{eqnarray*}
where in the last line we used point $(iii)$ of Proposition \ref{prop:moments}.
\end{proof}

\begin{lemma}
\label{lemma:g_v,c} 
For every partition $c$ of $\ell$, 
define the function  $g_{v,c}$ on $\R_+$ by
$$
g_{v,c}(x) \ := \ x^{|c|} e^{- v(\ell+1) x }\prod_{i = 1}^\ell \frac{ v(i)^{c(i)}}{c(i)!} . 
$$
For every $c_1, c_2\in\phi^{-1}(\ell)$, there exists $h$ such that for every $v= (v(1),\dots , v(\ell), v(\ell+1))\in \R_+^{\ell+1}$ with $ v(\ell+1)\geq \sum_{i \le \ell} v(i)$
\[ |\Cov(g_{v,c_1}(k p_1^{(k)}) , g_{v,c_2}(k p_2^{(k)}) )| \le  \max\{v(\ell +1), v(\ell+1)^2\} h(k) ,\]
where  $h$ is going to $0$ as $k\to \infty$.\end{lemma}
\begin{proof}
We use Lemma \ref{lemma-lipschitz}.
It is easy to see that $||g_{v,c}||_\infty \le 1$.
We just need to prove that the function $g_{v,c}$ is Lipshitz with constant $v(\ell +1) K$, for some $K$ that only depends on $c$.
 To compute the Lipshitz constant, set $\alpha:=v(\ell+1)$. Then 
$$g'_{v,c}(x)  = (|c| -\alpha x  )x^{|c|-1} e^{- \alpha x }\prod_{i = 1}^\ell \frac{v(i)^{c(i)}}{c(i)!}   $$
and 
$$g''_{v,c}(x)  = \bigg(-\alpha x+(|c| -\alpha x)(|c|-1 -\alpha x  )\bigg)x^{|c|-2} e^{- \alpha  x }\prod_{i = 1}^\ell \frac{v(i)^{c(i)}}{c(i)!},   $$
which switches sign when 
$|c|(|c|-1)-2|c|\alpha x+\alpha ^2x^2=0$, that is when $x=\frac{|c| +\epsilon \sqrt{|c|}}{\alpha }, \epsilon\in\{-1,1\}$. Since $\lim_{x\to\infty} g'_{v,c}(x)=0$ and $g'_{v,c}(0)=0$, the maximum of
$| g'_{v,c}|$ on $\R_+$ is attained for $\eps=+1$ or $\eps=-1$. Thus, using $v(i) \leq \alpha$, we have
\begin{eqnarray*}
\max_{x\in\R_+} |g'_{v,c}(x) | &=& \ \ \max_{-1,+1} \sqrt{|c|} (|c| +\eps \sqrt{|c|})^{|c|-1} e^{-(|c| +\eps \sqrt{|c|})}\frac{1}{\alpha^{|c|-1}}\prod_{i = 1}^\ell \frac{v(i)^{c(i)}}{c(i)!}\\
& \leq & \alpha \max_{-1,+1}
  e^{-(|c|+\eps\sqrt{|c|})} \sqrt{|c|} (|c|+\eps\sqrt{|c|})^{|c|-1} \ \times \  \frac{1}{c !},
\end{eqnarray*}
which completes the proof.
\end{proof}

\section*{Appendix B: Derivatives for Proposition \ref{prop1}.}
\label{appendix:rho}

In this section, the values of $x$ and ${\bf z}$ will remain fixed and for the sake of clarity, we will forget the dependence in those two constants in the next results.
\begin{lemma}\label{lem:faa}
Fix $x>0$, $\delta\in(0,1)$ and ${\bf z}\in \ell^1(\R_+)$. For every $\lambda\in [0,\delta]$, define
\[\rho(\lambda) \ := \ \E \bigg( x \exp(-{|\bz|\Gamma}/{x})  (\exp\left( \psi_\lambda(\bz)\Gamma /x\right) -1) \bigg). \]
Then $\rho$ is infinitely differentiable on $[0,\delta]$ and for every $\ell\in\N$, its $\ell^{th}$ derivative is such that
\begin{equation}\label{derivell}
\frac{1}{\ell !}  \rho^{(\ell)}(\lambda) \ = \   \sum_{c \in \phi^{-1}(\ell)}   \E\bigg( x \exp(-{|\bz|\Gamma}/{x})   \prod_{i=1}^\ell  \frac{(\Gamma/x)^{c(i)}}{c(i) !} \bigg) \prod_{i=1}^\ell \bigg(\sum_{k=i}^\infty \binom{k}{i} z(k) \lambda^{k-i} \bigg)^{c(i)} \end{equation}
with the convention  $0^0=1$. 
\end{lemma}
\begin{proof}
Let $f,g$ be infinitely differentiable on their respective spaces.
Fa\`a di Bruno's formula  
 states that
\begin{equation}\label{FdB}
\frac{1}{\ell !} \  \frac{d^\ell}{d\lambda^\ell}  \bigg( f\circ g\bigg)(\lambda) \ = \    \sum_{c \in \phi^{-1}(\ell)} f^{(|c|)}(g(\lambda))  \prod_{i=1}^\ell  \frac{1}{c(i)!}\bigg(\frac{g^{(i)}(\lambda)}{i !}\bigg)^{c(i)} .\end{equation}
Let $f(y)=e^y-1$ and $g(\lambda)=\psi_\lambda(\bz)\Gamma /x$ for some fixed value of $\Gamma$. Since ${\bf z}\in \ell^{1}(\R_+)$, the function $g$
is a power series with a radius of convergence larger or equal to $1$. Thus $g$ is infinitely differentiable on $[0,\delta]$  and 
\[  \  g^{(i)}(\lambda) =  \frac{\Gamma}{x}\sum_{k=i}^\infty  (k)_i z(k) \lambda^{k-i},\]
where $(k)_{i} = k \dots (k-i+1)$. 
For our choice of $f$ and $g$, \eqref{FdB} translates into the following formula
\[  \ \frac{1}{\ell !} \frac{d^\ell}{d\lambda^\ell}  \bigg( f\circ g\bigg)(\lambda)  \ = \ 
  \sum_{c \in \phi^{-1}(\ell)}  \bigg( x \exp\bigg((\psi_{\lambda}({\bz})-|\bz|){\Gamma}/{x}\bigg)  \prod_{i=1}^\ell  \frac{(\Gamma/x)^{c(i)}}{c(i) !} \bigg) \prod_{i=1}^\ell \bigg(\sum_{k=i}^\infty \binom{k}{i} z(k) \lambda^{k-i} \bigg)^{c(i)}. \]
  Recall that $\rho(\lambda)= \E(f\circ g(\lambda))$.
In order to complete the proof, we need to justify the derivation 
inside the expected value. Since the coefficients  of $\bz$ are non-negative 
then $\psi_{\lambda}({\bz})-|\bz| < 0$ for every $\lambda\in[0,1)$.
Furthermore, if $\lambda\in[0,\delta]$, then
\[ \psi_{\lambda}({\bz})-|\bz| \leq  \psi_{\delta}({\bz})-|\bz|=:-b. \]
As a consequence,
\[ \frac{1}{\ell !} \frac{d^\ell}{d\lambda^\ell}  \bigg( f\circ g\bigg)(\lambda) \ \leq \ 
 \sum_{c \in \phi^{-1}(\ell)}  \bigg( x \exp\bigg(-b {\Gamma}/{x}\bigg)  \prod_{i=1}^\ell  \frac{(\Gamma/x)^{c(i)}}{c(i) !} \bigg) \prod_{i=1}^\ell \bigg(\sum_{k=i}^\infty \binom{k}{i} z(k) \delta^{k-i} \bigg)^{c(i)}.\]
The right hand side of the inequality does not depend on $\lambda$, and one can check that it has a finite expectation. Since  $\rho(\lambda) = \E( f\circ g(\lambda) )$ (where the expectation is taken w.r.t. $\Gamma$), by standard derivation theorem under the integral,  $\rho(\lambda)$
is infinitely differentiable and the derivatives can be calculated by differentiating inside the expectation.
\end{proof}

\begin{lemma}
Let $\rho(\lambda)$ be the function introduced in Lemma \ref{lem:faa}. For every $\ell \in \N$ and $\lambda \in [0, \delta]$, with  $\delta <1$, we have
\[\frac{1}{\ell !} | \rho^{(\ell)}(\lambda)| \ < \   x \frac{(2-\delta)^{\ell-1}} {(1-\delta)^{2\ell}} .\]
\label{lem:mclaurin}
\end{lemma}
\begin{proof}
Fix $\ell \in \N$. 
We start by recalling that  for  $\lambda \in [0, \delta]$,
\[\sum_{k=i}^\infty \binom{k}{i}\lambda^{k-i}  =  \frac1{(1 - \lambda)^{i+1}}  \le  \frac1{(1 - \delta)^{i+1}}. \]
Observe that if $c\in  \phi^{-1}(\ell)$,  $\sum ic(i) = \ell$, so 
\begin{eqnarray*} \prod_{i=1}^\ell \bigg(\sum_{k=i}^\infty \binom{k}{i} z(k) \lambda^{k-i} \bigg)^{c(i)} & \le &\ \prod_{i=1}^\ell |\bz|^{c(i)}\left(  \frac1{(1 - \delta)^{i+1}} \right)^{c(i)}\\
  &\le& \frac{ |\bz|^{|c|} }{(1-\delta)^{\ell+|c|}}. \end{eqnarray*}
Finally, replacing in  \eqref{derivell}, 
\begin{eqnarray*}
\frac{1}{\ell !}  \rho^{(\ell)}(\lambda) & = &
\sum_{j  = 1}^\ell \sum_{\underset {|c| = j}{c \in \phi^{-1}(\ell)}}   \E\bigg( x \exp(-{|\bz|\Gamma}/{x})   \prod_{i=1}^\ell  \frac{(\Gamma/x)^{c(i)}}{c(i) !} \bigg) \prod_{i=1}^\ell \bigg(\sum_{k=i}^\infty \binom{k}{i} z(k) \lambda^{k-i} \bigg)^{c(i)} \\
&\le & \frac x{(1-\delta)^\ell} \sum_{j  = 1}^\ell \E\bigg(  \exp(-{|\bz|\Gamma}/{x}) (|\bz|\Gamma/x(1-\delta))^{j} \bigg) \sum_{\underset {|c| = j}{c \in \phi^{-1}(\ell)}} \prod_{i = 1}^\ell  \frac{1}{c(i)!} \\
& = & \frac x{(1-\delta)^\ell} \sum_{j  = 1}^\ell \E\bigg(  \exp(-{|\bz|\Gamma}/{x}) \frac{(|\bz|\Gamma/x(1-\delta))^{j}}{j!} \bigg) \binom{\ell -1}{j-1} \\
&\le &   \frac x{(1-\delta)^\ell}  \sum_{j  = 1}^\ell \binom{\ell -1}{j-1}\frac1{(1-\delta)^j} =  x\frac{(2-\delta)^{\ell-1}} {(1-\delta)^{2\ell}} .
\end{eqnarray*}
\end{proof}






\end{document}